\documentclass[11pt,leqno]{amsart}

\usepackage{hyperref}
\usepackage[all,cmtip]{xy}
\usepackage[active]{srcltx}
\usepackage{a4,latexsym}
\usepackage[utf8]{inputenc}
\usepackage{textcomp,fontenc}
\usepackage{exscale}
\usepackage{amsxtra,amsmath,amsfonts,amscd, amssymb, mathrsfs, amsthm}
\usepackage{color}

\newtheorem{thm}{Theorem}[section]
\newtheorem{prop}[thm]{Proposition}
\newtheorem{cor}[thm]{Corollary}
\newtheorem{lem}[thm]{Lemma}

\theoremstyle{definition}
\newtheorem{definition}[thm]{Definition}
\newtheorem{ex}[thm]{Example}
\newtheorem{rem}[thm]{Remark}
\newtheorem{art}[thm]{}

\numberwithin{paragraph}{section}
\numberwithin{equation}{thm}

\def\N{{\mathbb N}}
\def\Z{{\mathbb Z}}
\def\Q{{\mathbb Q}}
\def\R{{\mathbb R}}

\def\A{{\mathbb A}}

\def\G{{\mathbb G}}
\def\L{{\mathbb L}}

\def\KA{{\mathscr A}}

\def\KC{{\mathscr C}}
\def\KD{{\mathscr D}}

\def\KL{{\mathscr L}}
\def\KX{{\mathscr X}}
\newcommand{\metr}{{\|\hspace{1ex}\|}}
\newcommand{\Hom}{{\rm Hom}}

\def\an{{\rm an}}

\def\trop{{\rm trop}}
\def\Trop{{\rm Trop}}

\newcommand{\Xan}{{X^{\rm an}}}

\newcommand{\Tan}{{T^{\rm an}}}
\newcommand{\Uan}{{U^{\rm an}}}
\newcommand{\Acal}{{\mathscr A}}
\newcommand{\Bcal}{{\mathscr B}}
\newcommand{\Ccal}{{\mathscr C}}
\newcommand{\Dcal}{{\mathscr D}}

\newcommand{\Hcal}{{\mathscr H}}

\newcommand{\Lcal}{{\mathscr L}}
\newcommand{\Mcal}{{\mathscr M}}
\newcommand{\Ncal}{{\mathscr N}}
\newcommand{\Ocal}{{\mathscr O}}

\newcommand{\Ucal}{{\mathscr U}}

\newcommand{\Vcal}{{\mathscr V}}
\newcommand{\Wcal}{{\mathscr W}}
\newcommand{\Xcal}{{\mathscr X}}

\newcommand{\Spec}{{\rm Spec}}
\newcommand{\Spf}{{\rm Spf}}

\newcommand{\ve}{{\varepsilon}}
\newcommand{\id}{{\rm id}}

\newcommand{\supp}{{\rm supp}}

\newcommand{\relint}{{\rm relint}}
\newcommand{\kcirc}{{ K^\circ}}
\newcommand{\ktilde}{{ \tilde{K}}}

\newcommand{\Lan}{{L^{\rm an}}}

\newcommand{\val}{{\rm val}}

\newcommand{\dpa}{{d'_{\rm P}}}
\newcommand{\dpb}{{d''_{\rm P}}}

\title[Positivity properties of metrics and delta-forms]
{Positivity properties of metrics and delta-forms}

\author{Walter Gubler}
\address{W. Gubler, Fakultät für Mathematik, Universit{\"a}t Regensburg, 
93040 Regensburg, Germany}
\email{walter.gubler@mathematik.uni-regensburg.de}

\author{Klaus K{\"u}nnemann}
\address{K. K{\"u}nnemann, Fakultät für Mathematik, Universit{\"a}t Regensburg, 
93040 Regensburg, Germany}
\email{klaus.kuennemann@mathematik.uni-regensburg.de}

\date{\today}

\begin{document}

\begin{abstract}
In  previous work, we have introduced $\delta$-forms on the Berkovich analytification of an algebraic variety in order to study smooth or formal metrics via their associated Chern $\delta$-forms. In this paper, we investigate positivity properties of $\delta$-forms and $\delta$-currents. This leads to various plurisubharmonicity notions for continuous metrics on line bundles. In the case of a formal metric, we  show that many of these positivity notions are equivalent to Zhang's semipositivity. For piecewise smooth metrics, we prove that plurisubharmonicity can be tested on tropical charts in terms of convex geometry. We apply this to smooth metrics, to canonical metrics on abelian varieties and to toric metrics on toric varieties.
\bigskip

\end{abstract}

\maketitle
\setcounter{tocdepth}{1}
\setcounter{section}{-1}

\tableofcontents

\section{Introduction}\label{intro}
Pluripotential theory studies plurisubharmonic
functions and Monge-Amp\`ere operators 
and constitutes a central area of the modern theory of
complex analytic spaces.
In recent years a number of authors have introduced 
ideas and concepts from pluripotential theory 
into the theory of non-archimedean analytic  spaces.
Let us mention here the work of Baker and Rumely 
\cite{baker-rumely-2010} for the Berkovich
projective line and the work of Rumely, Kani, Chinburg-Rumely, Zhang, and Thuillier 
\cite{rumely-1989, kani-1989, chinburg-rumely-1993, zhang-1993,thuillier-thesis} 
via reduction graphs or skeletons for potential theory
on non-archimedean analytic curves.
An axiomatic  vision  of a theory of plurisubharmonic
functions on higher dimensional 
non-archimedean analytic spaces 
was formulated by Chinburg, Rumely and Lau in \cite[Sect. 4]{chinburgetal-2003}.

Higher dimensional theory started with Zhang's study of semipositive 
approximable metrics 
on line bundles in \cite{zhang-1995a} and in \cite{zhang-1995}. 
He realized that any model $\Lcal$ of a line bundle $L$ induces a 
metric on $L$. 
Such an algebraic metric is called semipositive if the restriction of $\Lcal$ to the special fibre is nef. 
A semipositive approximable metric on $L$ is defined as a uniform limit of metrics 
on $L$ for which some positive tensor powers are semipositive algebraic.
Bloch, Gillet and Soul\'e developed in 
\cite{bloch-gillet-soule-1995} a non-archimedean Arakelov theory based on the Chow groups of
all regular proper models over the valuation ring assuming resolution of singularities. 
Using Zhang's metrics and this non-archimedean Arakelov theory, Boucksom, Favre and Jonsson gave an approach to plurisubharmonic functions
via skeletons in the case of residue characteristic zero 
\cite{boucksometal-2012}.
If the involved measure is supported on a skeleton and if a certain algebraicity condition holds, then they prove the existence of 
a solution to the
non-archimedean Calabi-Yau problem \cite{boucksometal-2015}. Uniqueness was shown before by Yuan--Zhang \cite{yuan-zhang-2013} in complete generality. 
For proper toric varieties, Burgos, Philippon and Sombra gave in \cite{burgosetal-2011} a complete characterization of semipositive toric metrics in terms of convex functions and they also described in \cite{burgosetal-2012} positivity notions of toric metrics in terms of convex geometry. 

The theory of plurisubharmonic functions in \cite{boucksometal-2012} satisfies the required axioms of \cite{chinburgetal-2003}, except that it is not  of analytic character. 
Such an analytic theory has recently been established  
by Chambert-Loir and Ducros \cite{chambert-loir-ducros} introducing forms and currents on Berkovich spaces by 
using Lagerberg's superforms \cite{lagerberg-2012} on tropical charts. Their theory has the additional advantage that it works without any hypotheses on the characteristic. 
Their notion of plurisubharmonicity is directly related to the
positivity of currents. 
They have also transferred a part of the 
Bedford--Taylor theory to plurisubharmonic functions which are locally 
approximable by smooth plurisubharmonic functions leading to Monge--Amp\`ere measures for such locally psh-approximable functions. 
It would be desirable to extend the theory to all plurisubharmonic continuous functions as in 
 the complex Bedford--Taylor theory  \cite{bedford-taylor-1982} and  the monotone regularization theorem from the program in \cite[Sect. 4]{chinburgetal-2003} is also missing. 
 {Note that smooth metrics on line bundles have  first Chern forms, but algebraic metrics are singular {and} have only first Chern currents. The Bedford--Taylor theory for locally psh-approximable functions was used to define wedge products of such first Chern currents.}

The analytic theory of forms and currents in \cite{chambert-loir-ducros} has 
been extended in \cite{gubler-kuennemann}  
{to a new formalism of $\delta$-forms and $\delta$-currents. 
Both smooth
and algebraic metrics on line bundles have associated Chern $\delta$-forms leading to
natural wedge products {thus} bypassing the {use of} 
Bedford-Taylor theory.}
It is the aim of this paper to extend  the central notions of plurisubharmonicity and positivity 
to the theory of $\delta$-forms and $\delta$-currents.
We recall that a $\delta$-preform on $\R^r$ is a supercurrent $\alpha$ in the sense of Lagerberg \cite{lagerberg-2012} of the form 
$$\alpha = \sum_{i=1}^N \alpha_i \wedge \delta_{C_i},$$
where $\alpha_i$ is a superform and $\delta_{C_i}$ is the supercurrent of integration over a tropical cycle $C_i$ with smooth weights. The $\delta$-preforms on an open subset $\Omega$ of a tropical cycle $C$ with constant non-negative weights are supercurrents of the form $\alpha \wedge \delta_C$. Similarly as in complex analysis, Lagerberg \cite{lagerberg-2012} introduced weakly positive (resp. positive, resp. strongly positive)  superforms and supercurrents on $\R^r$. We show in Section \ref{posit} that this leads to corresponding positivity properties for supercurrents and hence for $\delta$-preforms on $\Omega$.

Let $X^\an$ denote the non-archimedean analytification of
a variety $X$ defined over an algebraically closed field $K$ complete with respect
to a non-trivial non-archimedean absolute value  $|\phantom{a}|$. 
 {{The hypotheses on $K$ make} some tropical and analytic arguments easier.
{They mean} no restriction of generality for positivity questions on varieties as we now explain: If one is not convinced to do analysis over a complete algebraically closed field as in the archimedean world, one might think about introducing $\delta$-forms on varieties over any non-archimedean field. However a sensible definition of positivity notions can always be checked after base change to a field of the above form and then by restricting to the irreducible components of the base change.}

 {By definition, very affine varieties {admit} a closed embedding into a multiplicative torus.}  
We apply the above to tropical charts $(V, \varphi_U)$ of $\Xan$, where $\varphi_U:U \rightarrow {\mathbb G}_m^r$ 
is the canonical closed embedding of a very affine open subset $U$ of $X$ and where $V:=\trop_U^{-1}(\Omega)$ for an open subset $\Omega$ of the tropical variety $\Trop(U)$. We have seen in \cite{gubler-kuennemann} that $\delta$-preforms on $\Omega$ represent generalized $\delta$-preforms on $V$ and by a sheafification process we get the bigraded algebra of generalized $\delta$-forms on any open subset $W$ of $\Xan$. The subalgebra of $\delta$-forms is characterized by a closedness condition with respect to natural differential operators $d',d''$ analogous to $\partial, \bar{\partial}$ in complex analytic geometry. As a topological dual, we have the space of $\delta$-currents on $W$ (see \cite[\S 4]{gubler-kuennemann}). The smooth forms from \cite{chambert-loir-ducros} build a subalgebra of the algebra of $\delta$-forms. 
We will show in Section \ref{positi} that the positivity notions of $\delta$-preforms induce corresponding positivity notions of (generalized) $\delta$-forms requiring that these notions are functorial. By duality, we get corresponding positivity properties for $\delta$-currents. 
Recall from \cite[\S 5.5]{chambert-loir-ducros} that a continuous function $f$ on $W$ is plurisubharmonic if the current $d'd''[f]$ is positive. In Section \ref{pluri}, we investigate variants of this definition. For example, we define $f$ to be $\delta$-psh if a similar condition holds for $\delta$-currents. It is not clear that these notions behave functorially with respect to morphisms $\varphi:X' \to X$ of algebraic varieties over $K$ and so we introduce also functorial psh (resp. functorial $\delta$-psh) continuous functions. In the following, we consider a line bundle $L$ of $X$ and a continuous metric $\metr$ on $L^{\rm an}$ over $W$. Then $\metr$ is called plurisubharmonic if $-\log \| s \|$ is psh on $W$ for any local frame $s$ of $L$. Similarly, we transfer the other positivity notions to metrics.

In Theorem \ref{liftvarieties}, we prove a crucial lifting result which enables us to lift closed subsets from the special fibre of a model over a  valuation ring to the generic fibre. In Section \ref{forma}, we recall the definition and some properties of a formal metric on a compact (reduced) strictly $K$-analytic space $V$.  If $V$ is a strictly $K$-analytic domain in the analytification of a proper variety $X$ over $K$, then we show in Corollary \ref{algebraic vs locally formal} that  any formal metric $\metr$ on the restriction of $L$ to $V$  extends to an algebraic metric of the line bundle $L$ over $X$. In Section \ref{semip}, we give a local variant 
of Zhang's semipositivity definition for a formal metric following a suggestion from Tony Yue Yu. Using our above lifting theorem, we prove  in Theorem \ref{main theorem for pl} the following result.

\begin{thm} \label{localnotes5}
Let $(L,\metr)$ be a formally metrized line bundle on a proper variety $X$ and let  $W$ be an open subset of $\Xan$. Then the following properties are equivalent:
\begin{itemize}
 \item[(1)] The formal metric $\metr$ is semipositive over $W$.
 \item[(2)] The restriction of the metric  {$\metr$} to $W$ is functorial $\delta$-psh.
 \item[(3)] The restriction of the metric $\metr$ to $W$ is functorial psh.
 \item[(4)] The $\delta$-form $c_1(L|_W,\metr)$ is positive on $W$.
 \item[(5)] The restriction of $\metr$ to $W \cap C^{\rm an}$  is psh for any closed curve $C$ of $X$.
\end{itemize}
\end{thm}

In Section \ref{Section: Semipositive approximable metrics}, we will show that a formal metric on $L$ is semipositive as a formal metric if and only if it a semipositive approximable metric. 
In the case of a discretely value field with residue characteristic $0$, this was first proved in \cite[Remark after Thm. 5.12]{boucksometal-2012}. Here, this is an easy consequence of our lifting theorem. We show also that the restriction of a semipositive approximable metric 
 to any closed curve is plurisubharmonic.  

In Section \ref{Piecewise smooth metrics}, we characterize piecewise smooth metrics in terms of convex geometry on tropical charts. 

\newtheorem*{main theorem 2}{Theorem \ref{psh slope condition}}
\begin{main theorem 2}
Let $L$ be a line bundle on  {an} algebraic variety $X$ over $K$.
Let $\metr$ be a piecewise smooth metric on  $L$ over an open subset $W$ of
$X^\an$. 
Then the metric $\metr$ is plurisubharmonic if and only if 
{for each tropical frame $(V,\varphi_U,\Omega,s,\phi)$ of $\metr$} we have
\begin{enumerate}
\item[(i)]
the restriction of $\phi$ to each maximal face of $\Trop (U) {\cap \Omega}$ is a convex function
and
\item[(ii)]
the corner locus $\tilde\phi\cdot \Trop (U)$ is an effective tropical cycle {on $\Omega$} 
with smooth weights.
\end{enumerate}
\end{main theorem 2}

Let us briefly explain the terminology used in the above theorem. The metric $\metr$ is called piecewise smooth if there is a tropical chart $(V,\varphi_U)$ in $W$ and a frame $s$ of $L$ over $U$ 
such that $-\log\|s\|= \phi \circ \trop_U$ on $V$ for a piecewise smooth function $\phi$ on the open subset $\Omega = \trop_U(V)$ of $\Trop(U)$. If $\Omega$ is also convex and if $\phi$ extends to a piecewise smooth function $\tilde\phi:N_\R \to \R$, then we call it  a tropical frame for $\metr$. The corner locus $\tilde\phi\cdot \Trop (U)$ is a tropical cycle of codimension $1$ in $N_{U,\R}$ which is supported in the non-differentiability locus of $\tilde\phi|_{\Trop(U)}$ and whose weights are defined in terms of the outgoing slopes (see \cite[Def. 1.10]{gubler-kuennemann}). In the remaining part of Section \ref{Piecewise smooth metrics}, we apply our results to compare  different positivity notions in the following situations:

\begin{itemize}
 \item $\delta$-metrics  in Proposition \ref{delta comparison thm};
 \item smooth metrics in Corollary \ref{smooth comparison thm};
 \item canonical metrics on abelian varieties in Example \ref{exampleabelianvarieties};
 \item canonical metrics on line
 bundles algebraically equivalent to zero
in Remark \ref{canonical metric on algebraic equivalent to zero};
 \item piecewise smooth toric metrics on toric varieties in Proposition \ref{torus prop}.
\end{itemize}

\subsection{Terminology}

We use $A \subseteq B$ to denote subsets and $B \setminus A$ for the complement of $A$ in $B$. The zero is included in $\N$, $\R_+$ and $\R_-$. In topology, compact means quasicompact and Hausdorff. 

The group of multiplicative units in a ring $A$ is denoted by $A^\times$.  
An (algebraic) variety over a field is an  {integral} scheme which is separated and of finite type. 
A curve is an algebraic variety of dimension one. 
A variety $U$ is called very affine if it has a closed 
immersion into a multiplicative torus. 
We refer to Section \ref{positi} for the canonical 
torus $T_U$ with cocharacter lattice $M_U$ and dual $N_U$, 
for the canonical closed embedding $\varphi_U: U \to T_U$ 
and for tropical charts. 
For Berkovich analytic spaces, we use the terminology from \cite{berkovich-ihes} and \cite{berkovich-book}. 

For the notation in convex geometry, we refer to \cite[Appendix A]{gubler-kuennemann}. We usually work with a finite dimensional real vector space $N_\R$ with an integral structure given by a lattice $N$. We recall that a polyhedron 
$\Delta$ is called integral $\R$-affine if the underlying linear space $\L_\Delta$ of the affine space $\A_\Delta$ generated by $\Delta$ is defined over $\Z$. An affine map $F$ is called integral $\R$-affine if the underlying linear map $\L_F$ is defined over $\Z$. For the terminology in tropical geometry, we refer to \cite[\S 1]{gubler-kuennemann}. We recall briefly that a tropical cycle $C=(\Ccal,m)$ in $N_\R$ consists of an integral $\R$-affine polyhedral complex $\Ccal$ and smooth weight functions $m_\Delta:\Delta \to \R$ for each maximal face $\Delta \in \Ccal$. We call $C$ {\it effective} on $\Omega \subseteq N_\R$ if $m_\Delta|_{\Delta \cap \Omega} \geq 0$ for each maximal $\Delta \in \Ccal$.

In the whole paper, $K$ is an algebraically closed field endowed with a complete non-trivial non-archimedean absolute value. We use $\kcirc$ for the valuation ring, $K^{\circ\circ}$ for its maximal ideal and $\ktilde=\kcirc/K^{\circ\circ}$ for the residue field. For a (formal) scheme $\Xcal$ over $\kcirc$, we use $\Xcal_\eta$ for the generic fibre and $\Xcal_s$ is the special fibre. 

\subsection{Acknowledgments}

The authors would like to thank Antoine Chambert--Loir, Antoine Ducros,
 {Philipp Jell,} and Tony Yue Yu for helpful conversations, 
 {the referee for numerous helpful comments} 
 and 
the collaborative research center
SFB 1085 funded by the Deutsche Forschungsgemeinschaft for its support.

\section{Positive forms on tropical cycles}\label{posit}

We recall some positivity notions for superforms from 
\cite{lagerberg-2012} and \cite{chambert-loir-ducros}.  
Then we introduce similar notions for $\delta$-preforms as 
defined in \cite{gubler-kuennemann}.
Let $N$ denote a free $\Z$-module of finite rank $r$ with
dual lattice $M$ and $N_\R=N\otimes_\Z\R$.
In this section $C = (\KC,m)$ always denotes an $n$-dimensional  tropical 
cycle on $N_\R$ with 
positive {\it constant} weights.

\begin{art} \label{positive superforms} \rm
We recall the positivity notions for superforms 
and supercurrents on an open 
subset $\tilde\Omega$ of $N_\R$ given in 
\cite[Sect. 2]{lagerberg-2012}, \cite[Sect. 5.1]{chambert-loir-ducros}.

(i) The canonical involution $J$ acts on the space $A(\tilde\Omega)$
of superforms on $\tilde\Omega$ and maps $(p,q)$-superforms
to $(q,p)$-superforms \cite[(1.2.5)]{chambert-loir-ducros}.
A superform $\alpha\in A^{p,p}(\tilde\Omega)$ is called \emph{symmetric}
if $J(\alpha)=(-1)^p\alpha$ holds and
\emph{anti-symmetric} if $J(\alpha)=(-1)^{p+1}\alpha$.

(ii) A \emph{positive $(r,r)$-superform} on $\tilde\Omega$ is given by 
\[
f d' x_1 \wedge d'' x_1  \wedge \dots \wedge d' x_r \wedge d'' x_r
\]
for a non-negative  function $f$ on $\tilde\Omega$, 
where $x_1, \dots, x_r$ is a basis of $M$.

(iii) A symmetric superform
$\alpha \in A^{p,p}(\tilde\Omega)$ is called {\it weakly positive} if 
\[
\alpha \wedge \alpha_1 \wedge J(\alpha_1) \wedge \dots \wedge 
\alpha_{r-p} \wedge J(\alpha_{r-p})
\]
is a positive $(r,r)$-superform for all $\alpha_1, \dots , \alpha_{r-p} \in A^{1,0}(\tilde\Omega)$.

(iv) A symmetric superform $\alpha \in A^{p,p}(\tilde\Omega)$ is called 
{\it positive} if 
\[
(-1)^{(r-p)(r-p-1)/2}\alpha \wedge \beta \wedge J(\beta)
\]
is a positive $(r,r)$-superform for every $(r-p,0)$-superform $\beta$ on $\tilde\Omega$.

(v) A superform $\alpha \in A^{p,p}(\tilde\Omega)$ is called
{\it strongly positive} if 
\[
\omega = \sum_{k=1}^l f_k \alpha_{k1} \wedge J(\alpha_{k1}) 
\wedge \dots \wedge \alpha_{kp} \wedge J(\alpha_{kp})
\]
for non-negative smooth functions $f_k$ and $(1,0)$-superforms $\alpha_{ki}$ on
$\tilde\Omega$. 
A strongly positive superform is automatically symmetric.

(vi) A supercurrent $\alpha\in D(\tilde\Omega)$ is called \emph{symmetric}
iff it vanishes on anti-symmetric superforms with compact support in $\tilde\Omega$.
A symmetric supercurrent $T\in D^{p,p}(\tilde\Omega)$ is called
\emph{(weakly/strong\-ly) positive} iff $\langle T,\alpha\rangle \geq 0$
for all (strongly/weakly) positive superforms $\alpha$ with compact support in $\tilde\Omega$.

 {For clarity, we explain here for the whole paper that 
{the phrasing \emph{(weakly/\-strong\-ly) positive}} includes three alternatives: It  defines a positive (resp.~a weakly positive, resp.~a strongly positive) supercurrent by evaluating  at positive (resp.~strongly positive, resp.~weakly positive) superforms of complimentary degree.} 

 (vii) Let $A^{p,p}_{+}(\tilde\Omega)$ be the space of positive $(p,p)$-superforms on $\tilde\Omega$. 
Similarly, we use 
$A^{p,p}_{\rm w+}(\tilde\Omega)$ (resp. $A^{p,p}_{\rm s+}(\tilde\Omega)$) 
for the space of weakly positive (resp. strongly positive) $(p,p)$-superforms on $\tilde\Omega$.
We denote by $D^{p,p}_{\rm (s/w)+}(\tilde\Omega)$ the 
corresponding spaces of supercurrents on $\tilde\Omega$.
\end{art}

\begin{prop} \label{positivity properties of superforms} 
\begin{itemize} 
\item[(a)] 
$A^{p,p}_{\rm s+}(\tilde\Omega) \subseteq A^{p,p}_{\rm + }(\tilde\Omega) 
\subseteq A^{p,p}_{\rm w+}(\tilde\Omega)$.
\item[(b)] 
For $p=0, 1, r-1, r$, we have equality everywhere in (a).
\item[(c)] 
The pull-back of a (weakly/strongly) positive superform with respect to an 
affine map is a (weakly/strongly) positive superform.
 \item[(d)] 
{The wedge product of a (weakly/strongly) positive superform with a
strongly positive superform is (weakly/strongly)  positive.}
\item[(e)]
A symmetric superform $\alpha\in A^{p,p}(\tilde\Omega)$ is 
(weakly/strongly) positive if and only if $\alpha\wedge \beta$  
is positive for every (strongly/weakly) positive  
superform $\beta$ of type $(r-p,r-p)$. 
\item[(f)]
There is a natural morphism $A(\tilde\Omega)\to D(\tilde\Omega)$
which maps a superform $\alpha$ to the associated supercurrent $[\alpha]$ 
 {given by $\langle [\alpha], \beta \rangle = \int_{\tilde\Omega} \alpha \wedge \beta$ for a compactly supported $\beta$ on $\tilde\Omega$.}  
A superform $\alpha\in A^{p,p}(\tilde\Omega)$ is (weakly/strongly) positive
if and only if $[\alpha]$ is (weakly/strongly) positive.
\end{itemize}
\end{prop}

\proof See 
\cite[\S 2]{lagerberg-2012} and \cite[\S 5.1]{chambert-loir-ducros}. 
For (e) and (f) observe in particular \cite[(5.1.2)]{chambert-loir-ducros}.
\qed

\begin{art}
(i) The notions and results mentioned above carry immediately over to
the situation where $\tilde\Omega$ is an open subset of an affine space 
under $N_\R$ \cite[5.1]{chambert-loir-ducros}.

(ii)  {The  notions \ref{positive superforms}(i)--(v) for a superform $\alpha$ can be defined  {at each point $x\in\tilde\Omega$ 
in the tensor product of exterior algebras of the cotangent space at $x$} 
and it is easy to see that 
$\alpha$ has such a property if and only if $\alpha(x)$ has the corresponding property for every $x \in \tilde\Omega$.}

(iii) For an open subset $\Omega$ of a polyhedron $\Delta$ in $N_\R$, a superform $\alpha$ on $\Omega$ is defined as the restriction of a superform $\tilde\alpha$ on an open subset $\tilde\Omega$ in $N_\R$ with $\tilde\Omega \cap \Delta = \Omega$.  
We call $\alpha$ \emph{(weakly/strongly) positive} if $\alpha$ is (weakly/strongly) positive at every point of $\Omega$ as a superform in the affine space $\A_\Delta$. An equivalent condition is that $\alpha|_{\Omega \cap \relint(\Delta)}$ is a (weakly/strongly) positive superform on the open subset $\relint(\Delta)$ of $\A_\Delta$. This uses the fact that the positivity loci of superforms are closed. 
\end{art}

\begin{ex}  \label{polyhedral supercurrent}
Let $\alpha\in D(\tilde\Omega)$ be a polyhedral supercurrent.
By definition (see \cite[Def. 2.3]{gubler-kuennemann}) 
there exists an integral $\R$-affine
polyhedral complex $\KD$ in $N_\R$ and a family 
$(\alpha_\Delta)_{\Delta\in\KD}$ of superforms $\alpha_\Delta$
on $\tilde\Omega\cap \Delta$ such that
\[
\alpha=\sum_{\Delta\in\KD}\alpha_\Delta\wedge\delta_\Delta
\mbox{ in }D(\tilde\Omega).
\]
An easy support argument shows that 
the $\alpha_\Delta$ are uniquely determined once we have fixed the
polyhedral complex $\KD$.
This implies that $\alpha$ is symmetric if and only if each superform
$\alpha_\Delta$ is symmetric.
It is furthermore a direct consequence of 
Proposition \ref{positivity properties of superforms} (e), (f)
that $\alpha$ is (weakly/strongly) positive if and only if 
each $\alpha_\Delta$ is (weakly/strongly) positive.
\end{ex}

\begin{ex} \label{convex function}
Let $\phi: N_\R \rightarrow \R$ be a piecewise smooth function. 
By definition, there is an integral $\R$-affine polyhedral complex $\KC$ with support $N_\R$ 
and a family $(\phi_\sigma)_{\sigma\in\KC}$ of
smooth functions $\phi_\sigma: \sigma\to \R$
such that $\phi|_{\sigma}=\phi_\sigma$ 
for all $\sigma\in\KC$.
The following properties are equivalent for a {convex} open subset $\tilde\Omega$ of $N_\R$:
\begin{enumerate}
\item[(i)]
The function $\phi$ is convex on $\tilde\Omega$.
\item[(ii)]
The supercurrent $d'd''[\phi|_{\tilde\Omega}]$ is positive on $\tilde\Omega$.
\item[(iii)]
Each function $\phi_\sigma$ is convex on $\tilde\Omega$ and the corner locus
$\phi\cdot N_\R$ is effective on $\tilde\Omega$.
\end{enumerate}
Here, the corner locus $\phi \cdot N_\R$ is a tropical cycle  on $N_\R$ of codimension $1$ which might be viewed as the tropical Weil divisor associated to the piecewise smooth function $\phi$ (see \cite[Def. 1.10]{gubler-kuennemann} for details). 
\end{ex}

\proof
The equivalence of (i) and (ii) is \cite[Prop. 2.5]{lagerberg-2012}.
We have 
\[
d'd''[\phi]=\sum_{\sigma}(d'd''\phi_\sigma)\wedge\delta_\sigma
+\delta_{\phi\cdot N_{\R}}
\] 
by \cite[Cor. 3.20]{gubler-kuennemann} where
$\sigma$ runs over the maximal polyhedra in $\KC$. 
It follows that the supercurrent $d'd''[\phi]$ is polyhedral.
Then the equivalence of (ii) and (iii) follows from  Example \ref{polyhedral supercurrent}.
\qed

\begin{art} \label{delta-preformen euklidisch}
Let $\tilde\Omega$ be an open subset of $N_\R$. Recall from \cite[Def. 2.9]{gubler-kuennemann} that a $\delta$-preform on $\tilde\Omega$ is 
a supercurrent $\alpha \in D(\tilde\Omega)$ of the form
$\alpha = \sum_{i \in I} \alpha_i \wedge \delta_{C_i}$, 
where $I$ is a finite set, $\delta_{C_i}$ is the supercurrent of integration over a tropical cycle $C_i$ in $N_\R$ with smooth weights and $\alpha_i$ is a superform on $\tilde\Omega$. The wedge product of superforms and the tropical intersection product makes $P^{\cdot, \cdot}(\tilde\Omega)$ into a bigraded algebra \cite[Prop. 2.12]{gubler-kuennemann}. By definition, the bigrading and the positivity notions are induced by the corresponding notions in $D^{\cdot, \cdot}(\tilde\Omega)$.
Note that every $\delta$-preform is a polyhedral supercurrent (see Example \ref{polyhedral supercurrent}).

Explicitly, the $\delta$-preform $\alpha$ has bidegree $(p,q)$ if and only if $\alpha_i \in A^{p_i,q_i}(\tilde\Omega)$ and $C_i$ is a tropical cycle of codimension $l_i$ in $N_\R$ with $p_i+l_i=p$ and $q_i + l_i = q$ for all $i \in I$. We say that the $\delta$-preform $\alpha$ has {\it codimension}  $l$ if there is a decomposition with all $C_i$ of codimension $l$. We define a trigrading  $P^{s,t,l}(\tilde\Omega)$ considering all $\delta$-preforms $\alpha$ as above with $p_i=s$, $q_i=t$ and $l_i=l$.  
\end{art}

In the following, we consider an open subset $\Omega$ of the support $|\Ccal|$ for the $n$-dimensional tropical cycle $C=(\KC,m)$ in $N_\R$ with positive constant weights. The goal is to transfer the above notions to this relative situation. 

\begin{art} \label{delta-preformen}
We choose an open subset $\tilde\Omega$ in $N_\R$ with $\Omega=\tilde\Omega\cap |\KC|$.  
The bigraded algebra of superforms on $\Omega$ is defined by $A(\Omega):=
\{ {\eta|_{\Omega}} \mid  {\eta} \in A(\tilde\Omega)\}$, the subalgebra of superforms with compact support in $\Omega$ is denoted by $A_c(\Omega)$.

 {A linear functional  $T \in \Hom_\R(A_c(\Omega),\R)$ is called a {\it supercurrent on  $\Omega$} if there is $\tilde{T} \in D(\tilde\Omega)$ such that $T(\eta|_\Omega) = \tilde{T}(\eta)$ for all $\eta \in A_c(\tilde\Omega)$. We will identify $D(\Omega)$ with a subspace of $D(\tilde\Omega)$ using the canonical map $T \mapsto \tilde{T}$.
}

A partition of unity argument shows that these definitions do not depend on the choice of $\tilde\Omega$ and the same holds for all definitions below (see \cite[3.1-3.3]{gubler-kuennemann} for details). Note that on $A(\Omega)$ and dually on $D(\Omega)$, we have canonical differential operators $d',d''$ which are analogues of $\partial$ and $\bar\partial$ in complex analysis (see \cite[1.4.7]{chambert-loir-ducros} or \cite[3.2]{gubler-forms}). 

 {(i) For $\tilde\alpha \in P(\tilde\Omega)$, we note that $\tilde\alpha\wedge \delta_C$ is a polyhedral supercurrent as in Example \ref{polyhedral supercurrent} such that the supporting polyhedra are contained in  $|\KC|$ and hence $\tilde\alpha\wedge \delta_C$ is contained in the subspace $D(\Omega)$ of $D(\tilde\Omega)$. Such elements of $D(\Omega)$ are called {\it $\delta$-preforms on $\Omega$} and $P(\Omega)$ denotes the space of $\delta$-preforms on $\Omega$.}

(ii) We give $P(\Omega)$ the unique structure as a bigraded algebra such that the surjective map $P(\tilde\Omega) \to P(\Omega)$ is a homomorphism of bigraded algebras. Similarly, we define the grading by {\it codimension} on $P(\Omega)$ and the trigrading $P^{s,t,l}(\Omega)$ with $l$ the codimension and $(p,q)=(s+l,t+l)$ the type of the $\delta$-preform. 

(iii) A $\delta$-preform $\alpha\in P^{p,p}(\Omega)$ is called 
{\it (weakly/strongly) positive} if and only if $\alpha$ induces
a (weakly/strongly) positive supercurrent in $D^{p,p}(\tilde\Omega)$.

(iv)
Note that $\delta$-preforms on $\Omega$ of codimension $0$ are the same as 
superforms on $\Omega$. In particular, the above gives the corresponding positivity notions for superforms on $\Omega$.
\end{art}

\begin{prop} \label{positivity properties of preforms} 
For an open subset $\Omega$ of $|\KC|$, we have the following properties:
\begin{itemize} 
 \item[(a)] 
$P^{p,p}_{\rm s+}(\Omega) \subseteq P^{p,p}_{\rm + }(\Omega) \subseteq P^{p,p}_{\rm w+}(\Omega)$.
 \item[(b)] 
For $p=0,1,n-1,n$, equality holds everywhere in (a) and more precisely
\begin{eqnarray*}
P^{t+l,t+l}_{\rm s+}(\Omega)\cap P^{t,t,l}(\Omega) 
&=& P^{t+l,t+l}_{\rm + }(\Omega) \cap P^{t,t,l}(\Omega) \\
&=& P^{t+l,t+l}_{\rm w+}(\Omega)\cap P^{t,t,l}(\Omega) 
\end{eqnarray*}
for all $l$ and $t=0,1,n-l-1,n-l$.
\end{itemize}
If $\Omega$ is an open subset of $N_\R$ (and $C=N_\R$ with weight $1$), then the following 
additional properties hold:
\begin{itemize}
\item[(c)] 
The pull-back of a (weakly/strongly) positive $\delta$-preform on $\Omega$ 
with respect to an affine map is a (weakly/strongly) positive $\delta$-form.
 \item[(d)] 
{The wedge product of a strongly positive $\delta$-preform with 
a (weakly/strong\-ly) positive $\delta$-preform on $\Omega$ 
is (weakly/strongly) positive.}
 \end{itemize}
\end{prop}

\proof 
Properties (a) and (b) follow immediately from Proposition 
\ref{positivity properties of superforms}. 
If $\Omega$ is an open subset of $N_\R$, then the stable intersection product 
and the pull-back with respect to affine maps are well-defined as $\delta$-preforms. 
To prove (c) (resp. (d)), we use also (2.12.3) (resp. (2.12.5)) in 
\cite{gubler-kuennemann}. \qed

\begin{lem} \label{positivity of preforms and covering}
Let $F:N_\R' \rightarrow N_\R$ be an integral $\R$-affine map and  
 {$C'=(\KC',m')$}  an effective $n$-dimensional tropical cycle on $N_\R'$ 
with constant weights.
 {Write $F_*(C')=:C=(\KC,m)$ and} consider an open subset $\tilde\Omega$ of 
$N_\R$ and $\tilde\beta \in P^{p,p}(\tilde\Omega)$. 
For $\Omega := \tilde\Omega \cap |\KC|$, the $\delta$-preform 
$\beta := \tilde\beta \wedge \delta_C$ is (weakly/strongly) positive on $\Omega$ 
if $\beta':=F^*(\tilde\beta) \wedge \delta_{C'}$ is (weakly/strongly) 
positive on $\Omega':=F^{-1}(\Omega) \cap | {\KC'}|$. 
\end{lem}

\begin{proof} We are here essentially in the setup of the projection formula 
given in \cite[Prop. 2.14]{gubler-kuennemann}. 
We may assume without loss of generality that $\tilde\beta$
is a $\delta$-preform of pure codimension $l$. 
Hence we can write
\[
\tilde\beta=\sum_{i\in I}\alpha_i\wedge\delta_{C_i}
\]
for suitable $\alpha_i\in A^{p-l,p-l}(\tilde\Omega)$ and tropical cycles
$C_i$ of codimension $l$ in $N_\R$.
After suitable refinements we may assume that 
$F_*(\KC')$ is a polyhedral subcomplex of $\KC$ and that
$\KC$ and $\KC'$ are
polyhedral complexes of definition for $\tilde\beta$ and $\tilde\beta'$.
We get polyhedral decompositions 
\begin{eqnarray*}
&\beta=\tilde\beta\wedge \delta_C=\tilde\beta\wedge \delta_{F_*C'}
=\sum\limits_{\sigma\in \KC_{n-l}}\alpha_\sigma\wedge\delta_\sigma,&\\
&\beta'=F^*(\tilde\beta)\wedge\delta_{C'}=\sum\limits_{\sigma'\in \KC'_{n-l}}\alpha_{\sigma'}\wedge\delta_{\sigma'}&
\end{eqnarray*}
as in \cite[(2.14.3), (2.14.4)]{gubler-kuennemann}. 
Given $\sigma\in \KC_{n-l}$, we have as in \emph{loc. cit.}
\begin{equation}\label{pre_projection_formula}
\alpha_\sigma=\sum\limits_{
\genfrac{}{}{0pt}{}{\sigma'\in \KC_{n-l}'}{F(\sigma')=\sigma}}
\bigl[N_\sigma:\L_F(N_\sigma')\bigr]\cdot \tilde\alpha_{\sigma'}
\end{equation}
where $\tilde\alpha_{\sigma'}$ denotes the unique superform in 
$A_\sigma(\sigma\cap\tilde\Omega)$ such that
$F^*(\tilde\alpha_{\sigma'})=\alpha_{\sigma'}$ in
$ A_{\sigma'}(\sigma'\cap F^{-1}(\tilde\Omega))$.
Now assume that $\beta'$ is positive.
This implies by definition that all the $\alpha_{\sigma'}$ are
positive.
If $F(\sigma')=\sigma$ then $F$ induces an isomorphism from $\sigma'$ to 
$\sigma$ and $\tilde\alpha_{\sigma'}$ is positive.
Then formula \eqref{pre_projection_formula} implies that $\alpha_\sigma$
is positive as well.
In the same way one proves the variants  for weak and strong positivity.
\end{proof}

\begin{definition} \label{supercurrents} \rm
Recall that the tropical cycle $C=(\KC,m)$ has dimension $n$. 
Let $\Omega$ be an open subset of $|\KC|$. 
A supercurrent $T \in D^{p,p}(\Omega)$ is called {\it (weakly/strongly) positive} if 
$
\langle T, \alpha \rangle \geq 0
$
holds for all (strongly/weakly) positive superforms $\alpha \in A_c^{n-p,n-p}(\Omega)$. 
The corresponding spaces of (weakly/strongly) positive supercurrents are 
denoted by $D_{+}^{p,p}(\Omega)$, $D_{\rm w+}^{p,p}(\Omega)$ and $D_{\rm s+}^{p,p}(\Omega)$.
\end{definition}

\begin{rem} \label{positive delta-forms vs supercurrents}
From Proposition \ref{positivity properties of preforms}, we obtain immediately
\[
D^{p,p}_{\rm s+}(\Omega) \subseteq D^{p,p}_{\rm + }(\Omega) \subseteq D^{p,p}_{\rm w+}(\Omega)
\]
with equalities for $p=0,1,n-1,n$.

A superform on $\Omega$ is (weakly/strongly) positive if and only if its associated supercurrent is (weakly/\-strongly) positive in $D(\Omega)$. More generally,
a $\delta$-preform on $\Omega$ is (weakly/strongly) positive if and only if its associated supercurrent has the same positivity property in $D(\Omega)$. The proof is similar to the  proof of Proposition \ref{functorial current positive and positive delta-forms} and we leave the details to the reader. 
\end{rem}

\begin{rem}\label{posnegtrop}
(i) Let $\tilde\Omega$ be an open subset of $N_\R$ and $p\in \N$.
If $\alpha\in A^{p,p}(\tilde\Omega)$ is a superform
such that $\alpha$ and $-\alpha$ are weakly positive, then  
$\alpha=0$. This follows from the fact that every symmetric superform in $A^{p,p}(\tilde\Omega)$ is 
the difference of two strongly positive elements in $A^{p,p}(\tilde\Omega)$ \cite[Lemme 5.2.3]{chambert-loir-ducros} 
and from the duality in Proposition \ref{positivity properties of superforms}(e).

(ii) 
Recall that $\Omega$ is an open subset of $|\KC|$ for the effective tropical cycle $C=(\KC,m)$ with constant weights. 
We consider a $\delta$-preform $\alpha \in P^{p,p}(\Omega)$. 
After a subdivision of $\KC$ we may write 
\begin{equation}\label{polreppos}
\alpha=\sum_{\Delta \in \KC} \alpha_\Delta \delta_\Delta \in D(\Omega)
\end{equation}
with superforms $\alpha_\Delta$ on the open subsets 
$\Omega \cap \Delta$ of $\Delta$. 
The representation 
\eqref{polreppos} is unique up to subdivision. 
Now assume that  $\alpha$ and $-\alpha$
are weakly positive. Then each $\alpha_\Delta$ must
be weakly positive and (i) implies $\alpha=0$.
\end{rem}

\section{Positive delta-forms and delta-currents} \label{positi}

Let $K$ be  an algebraically closed field endowed with a complete non-trivial non-archimedean absolute value $|\phantom{a}|$. In the following, we will always work in this setup except in Section \ref{lifti}.

Let $X$ be an algebraic variety of dimension $n$ over $K$. 
For an open subset $W$ of the Berkovich space $\Xan$, we use the bigraded algebra of generalized $\delta$-forms $P^{\cdot,\cdot}(W)$ and its bigraded subalgebra 
$B^{\cdot,\cdot}(W)$ of $\delta$-forms. The latter is a differential algebra with respect to differential operators 
$d',d''$ and behaves similarly as the algebra of differential forms on a complex manifold with respect to $\partial,\bar \partial$ (see \cite[\S 4]{gubler-kuennemann} for details). 
As a topological dual of $B_c(W)$, we have  the space of $\delta$-currents $E(W)$ (see \cite[\S 6]{gubler-kuennemann}), where  the subscript $c$ means always compact support. 
The smooth forms from  \cite[\S 3]{chambert-loir-ducros} give a bigraded differential subalgebra $A^{\cdot,\cdot}(W)$ of $B^{\cdot,\cdot}(W)$ inducing a canonical linear map from $E(W)$ to the space of currents $D(W)$, where the latter is defined as a topological dual of $A_c(W)$ in \cite[\S 4]{chambert-loir-ducros}. 

The goal of this section is to transfer the positivity notions from Section \ref{posit} to the sheaves $B^{\cdot,\cdot}, P^{\cdot,\cdot}, E^{\cdot,\cdot}$ and to compare them with the positivity notions on $A^{\cdot,\cdot}, D^{\cdot,\cdot}$ introduced in  \cite[\S 5]{chambert-loir-ducros}

\begin{art} \label{building block of delta-preform}
We start with a {\it tropical chart} $(V,\varphi_U)$ on $\Xan$. Recall from \cite[4.15]{gubler-forms} that this is a very affine open subset $U$ of $X$ with a canonical closed embedding $\varphi_U:U \rightarrow T_U$ into the torus $T_U$ with the character lattice $M_U = \Ocal(U)^\times/K^\times$ and an open subset $V:=\trop_U^{-1}(\Omega)$ for an open subset  {$\Omega$} of the tropical variety $\Trop(U)=\trop_U(\Uan)$. Here, we have used   the tropicalization map $\trop_U:\Uan \to N_{U,\R}$ for the cocharacter lattice $N_U=\Hom(M_U,\Z)$. Note that tropical charts form a basis of topology for $\Xan$ \cite[Prop. 4.16]{gubler-forms}.

 {In \cite[4.4]{gubler-kuennemann}, the bigraded algebra $P^{\cdot,\cdot}(V,\varphi_U) := P^{\cdot,\cdot}(\tilde\Omega)/N^{\cdot,\cdot}(V,\varphi_U)$ was defined for an open subset $\tilde\Omega$ of $N_{U,\R}$ with $\Omega = \tilde\Omega \cap \Trop(U)$, where $N^{\cdot,\cdot}(V,\varphi_U)$ includes the kernel of the homomorphism $P(\tilde\Omega)\to P(\Omega)$ from \ref{delta-preformen}(i) and the precise definition 
 {in \emph{loc.~cit.~}}takes this into account in a functorial way.} 
In fact, every 
generalized $\delta$-form 
 is locally given by elements $\beta_U$ in such a  $P^{\cdot,\cdot}(V,\varphi_U)$. If $\beta_U$ is represented by a $\delta$-preform of codimension $l$, then we say that $\beta_U$ has {\it codimension} $l$. The grading by codimension $l$ leads to subspaces $P^{s,t,l}(U,\varphi_U)$ of $P^{s+l,t+l}(U,\varphi_U)$ and hence to subspaces $P^{s,t,l}(W)$ of $P^{s+l,t+l}(W)$. 
\end{art}

\begin{definition}\label{positivity on tropical charts}   
For every morphism $f:X'\to X$ of varieties over $K$ and any  pair of
charts $(V',\varphi_{U'})$ on $X'^\an$ and $(V,\varphi_U)$ on $X^\an$ with $f(U') \subseteq U$ and $f(V') \subseteq V$, 
we have a pull-back $f^*:P^{p,p}(V,\varphi_U) \rightarrow P^{p,p}(V',\varphi_{U'})$ 
 {induced by the canonical affine map $F:N_{U',\R}\to N_{U,\R}$. We use also the 
 restriction map
$P^{p,p}(V',\varphi_{U'}) \rightarrow P^{p,p}(\Omega')$, 
$\beta'\mapsto \beta'|_{\Omega'}$} 
to $\delta$-preforms on  $\Omega' := \trop_{U'}(V')$  {which is induced by wedge product with $\delta_{\Trop(U')}$.}

(i) We say that $\beta \in P^{p,p}(V,\varphi_U)$ is 
{\it (weakly/strongly) positive} if $f^*(\beta)|_{\Omega'}$ 
is a (weakly/strongly) positive $\delta$-preform on 
 {the open subset $\Omega'$ of the tropical cycle $\Trop(U')$ in the sense of \ref{delta-preformen}} 
for every $f:X' \to X$ and every  $(V',\varphi_{U'})$ as above. 
These forms yield subspaces 
\begin{equation}\label{subspacesforP}
P^{p,p}_{\rm s+}(V,\varphi_U) \subseteq P^{p,p}_{\rm + }(V,\varphi_U) \subseteq P^{p,p}_{\rm w+}(V,\varphi_U)
\end{equation}
of $P^{p,p}(V,\varphi_U)$ with equality in the cases $p=0,1,n-1,n$. Obviously, all these positivity notions 
are stable with respect to pull-back $f^*:P^{p,p}(V,\varphi_U) \rightarrow P^{p,p}(V',\varphi_{U'})$ for any $f$ as above.

 {(ii)} 
 {We say that $\beta \in P^{p,p}(V,\varphi_U)$ is \emph{represented by}
$\tilde\beta\in P^{p,p}(\tilde\Omega)$ if $\beta$ is the class of
$\tilde\beta$ in
$P^{p,p}(V,\varphi_U)= P^{p,p}(\tilde\Omega)/N^{p,p}(V,\varphi_U)$.}
We say that $\beta \in P^{p,p}(V,\varphi_U)$ is {\it (weakly/strongly) 
positively {representable}} if there exists an open subset
$\tilde\Omega$ of  {$N_{U,\R}$} with  $\Omega=\tilde\Omega\cap \Trop(U)$ and a 
(weakly/ strongly) positive $\tilde{\beta} \in P^{p,p}(\tilde\Omega)$ 
\emph{representing} $\beta$ . 
{In this case we call $\tilde\beta$ a
\emph{positive representative of $\beta$}}.
\end{definition}

\begin{art}
 {If $\beta$ and $\tilde\beta$ are as in (ii), then $f^*(\beta)$ is represented by $F^*(\tilde\beta)$ and hence $f^*(\beta)|_{\Omega'}= F^*(\tilde\beta) \wedge \delta_{\Trop(U')}$ on $\Omega'$.  By Proposition \ref{positivity properties of preforms}, a (weakly/strongly) positively representable element
$\beta \in P^{p,p}(V,\varphi_U)$ is  (weakly/\-strongly) positive.}

 {Similarly, we see that the notion of (weakly/strongly) positive
representability  
in $P^{p,p}(V,\varphi_U)$ is closed under  pull-back  
and that the wedge-product 
of a strongly positively representable element with a (weakly/strongly) positively representable element  is (weakly/strongly) positively 
representable.}  
\end{art}

\begin{lem} \label{positivity on tropical charts and coverings}
Let $f:X' \rightarrow X$ be a generically finite dominant morphism of 
varieties over $K$ , let $(V,\varphi_U)$ be a tropical chart on $\Xan$ 
and let $U'$ be a very affine open subset of $X'$ with $f(U') \subseteq U$. 
Then $(V',\varphi_{U'})$ is a tropical chart on $X'^\an$ for $V':=f^{-1}(V) \cap U'^\an$. 
Moreover,  $\beta \in P^{p,p}(V,\varphi_U)$ is (weakly/strongly) positive 
if and only if  $f^*(\beta) \in P^{p,p}(V',\varphi_{U'})$ is (weakly/strongly) positive. 
\end{lem}

\proof 
Let $F:N_{U',\R} \to N_{U,\R}$ be the canonical integral $\R$-affine map induced by $f:U'\to U$. Then 
$\Omega':=F^{-1}(\Omega) \cap \Trop(U')$ is an open subset of $\Trop(U')$ and functoriality of tropicalizations shows 
that $V'=\trop_{U'}^{-1}(\Omega')$. We conclude that $(V',\varphi_{U'})$ is a tropical chart on $X'^\an$. 
The Sturmfels--Tevelev multiplicity formula shows 
\begin{equation} \label{Sturmfels-Tevelev}
F_*(\Trop(U')) = \deg(f)\Trop(U)
\end{equation}
(see \cite{sturmfels-tevelev},
\cite{baker-payne-rabinoff}
or \cite[Thm. 13.17]{gubler-guide}). Using also that our positivity notions are stable under pull-back,  
the last claim follows from 
 Lemma \ref{positivity of preforms and covering}  \qed

\begin{rem}\label{posnegtropchart}
{Let $\alpha\in P^{p,p}(V,\varphi_U)$  such that
$\alpha$ and $-\alpha$ are weakly positive.
Then $\alpha=0$.
This follows from Remark \ref{posnegtrop} applied to all compatible pairs 
of charts as in \ref{positivity on tropical charts}.}
\end{rem}

In the following, $W$ is an open subset of $\Xan$. We introduce the above positivity notions on the space $P^{p,p}(W)$ of generalized $\delta$-forms on $W$. 

\begin{definition} \label{positive delta-forms on W}
A generalized $\delta$-form 
$\beta \in P^{p,p}(W)$ is called {\it (weakly/strongly) positive} if 
at  any given point of $W$ there exists a tropical
chart $(V,\varphi_U)$ such that $V\subseteq W$ and $\beta|_V=\trop_U^*(\beta_U)$
for a (weakly/strongly) positive element $\beta_U\in P^{p,p}(V,\varphi_U)$. Note that such a $\beta_U$ is uniquely determined by $(V,\varphi_U)$  \cite[Prop. 4.18]{gubler-kuennemann}. 
These generalized $\delta$-forms define subspaces 
\begin{equation}\label{Pfiltration}
P^{p,p}_{\rm s+}(W) \subseteq P^{p,p}_{\rm + }(W) \subseteq P^{p,p}_{\rm w+}(W)
\end{equation}
of $P^{p,p}(W)$. 
For $p=0,1,n-1,n$, we have equalities in the above chain.

Similarly, we define  {\it (weakly/strongly) 
positively {representable}} generalized $\delta$-forms in  $P^{p,p}(W)$. {For $p=0,1$, these three positivity notions agree again. All  six positivity notions are  closed under pull-back.}
\end{definition}

\begin{prop} \label{local positivity is fine}
Let $\beta\in P^{p,p}(W)$ be a (weakly/strongly) positive
generalized $\delta$-form on an open subset $W$ of $X^\an$.
Let $(V,\varphi_U)$ be a tropical chart of $X^\an$ such that $V\subseteq W$
and $\beta|_V$ is induced by $\beta_U\in P^{p,p}(V,\varphi_U)$.
Then $\beta_U$ is (weakly/strongly) positive. 
\end{prop}

\begin{proof}
First, we note that if $U'$ is a very affine open subset of $X$ and if $\beta$ is given on tropical charts 
$(V_j',\varphi_{U'})_{j \in J}$ in $W$ by a (weakly/strongly) positive $\beta_j \in P^{p,p}(V_j',\varphi_{U'})$, then 
$\beta$ is given on the tropical chart $(V':= \bigcup_{j\in J} V_j',\varphi_{U'})$ by a unique (weakly/strongly) positive $\beta_{U'} \in P^{p,p}(V',\varphi_{U'})$. 
Existence and uniqueness follow from \cite[Prop. 4.12]{gubler-kuennemann}. Positivity follows 
from the fact that the positivity notions of $\delta$-preforms are defined locally using that we always have the same tropicalization map $\trop_{U'}$. 

Using this property and properness of $\trop_U$, we may assume that $V$ is relatively compact. Therefore $V$ may be covered by finitely many tropical charts $(V_i,\varphi_{U_i})$, $i=1,\dots,s$, such that $\beta$ is given on any $V_i$ by a (weakly/strongly) positive $\beta_i \in P^{p,p}(V_i,\varphi_{U_i})$. Then $U':= U \cap U_1 \cap \dots \cap U_s $ is a very affine open subset of $X$ \cite[4.13]{gubler-forms}. For $V':= U'^\an \cap \bigcup_{i=1}^s V_i$, we get a tropical chart $(V',\varphi_{U'})$ as in the proof of \cite[Prop. 5.13]{gubler-forms}.  Again using the property at the beginning of the proof, we deduce that $\beta$ is given on $(V',\varphi_{U'})$ by the unique (weakly/strongly) positive $\beta_{U'}\in P^{p,p}(V',\varphi_{U'})$ which agrees with $\beta_i$ on $U'^\an \cap V_i$. The restriction of $\beta_{U'}$ to the tropical subchart $(V \cap U'^\an, \varphi_{U'})$ remains (weakly/strongly) positive. Since $U' \subseteq U$, Lemma \ref{positivity on tropical charts and coverings} proves the claim.
\end{proof}

\begin{rem}\label{compare cld}
Chambert-Loir and Ducros have introduced subspaces 
\begin{equation}\label{Afiltration}
A^{p,p}_{\rm s+}(W) \subseteq A^{p,p}_{\rm + }(W) \subseteq A^{p,p}_{\rm w+}(W)
\end{equation}
of strongly-positive, positive, and weakly positive 
smooth forms.
Since they use analytic moment maps, we would like to rephrase their
definition in terms of tropical charts and the language of
\cite{gubler-forms}.
A smooth form $\alpha\in A^{p,p}(W)$ is 
(weakly/strongly) positive if whenever
$(V,\varphi_U)$ is a tropical chart of $X^\an$ such that $V\subseteq W$
and $\alpha|_V$ is induced by $\alpha_U\in A^{p,p}(\Omega)$
for $\Omega=\trop_U(V)$ then $\alpha_U$ is (weakly/strongly) positive,
i.e. the restriction of the superform $\alpha_U$
to any face of $\Trop(U)$ is (weakly/strongly) positive.
This follows from 
\cite[Prop. 7.2]{gubler-guide}. 
Hence we get \eqref{Afiltration}
as the intersection of \eqref{Pfiltration} with $A^{p,p}(W)$.
\end{rem}

\begin{prop}\label{positivity representable products}
On an open subset $W$ of $\Xan$, the following holds:
\begin{itemize} 
 \item[(a)] The product of a strongly positively representable generalized $\delta$-form with a (weakly/strongly) positively representable generalized $\delta$-form is a (weakly/strongly) positively representable generalized $\delta$-form. 
 \item[(b)] The product of a strongly positively representable generalized $\delta$-form with a (weakly/strongly) positive generalized $\delta$-form is a (weakly/strongly) positive generalized $\delta$-form. 
 \item[(c)] The product of a (weakly/strongly) positively representable generalized 
$\delta$-form of type $(p,p)$ with a (strongly/weakly) positive generalized 
$\delta$-form of type $(n-p,n-p)$ is a positive generalized 
$\delta$-form.
\end{itemize}
We can replace positively representable by positive in   (b) and (c) if 
at least one of the two factors is a smooth form. 
\end{prop}

\begin{proof} 
Let $(V,\varphi_U)$ be a tropical chart in $W$. 
It is enough to show the properties (a) and (b) for $\alpha \wedge \beta$ with $\alpha \in P^{p,p,l}(V,\varphi_U)$ and $\beta \in P^{p',p',l'}(V,\varphi_U)$.  Then (a) follows from Proposition \ref{positivity properties of preforms}(d). 

For (b), we choose $\delta$-preforms $\tilde\alpha, \tilde\beta$ on an open subset $\tilde\Omega$ of $N_{U,\R}$ which represent $\alpha,\beta$, where $\Omega=\trop_U(V)= \tilde\Omega \cap \Trop(U)$. Since the positivity notions are functorial, it is enough to show that 
$\alpha \wedge \beta |_{\Omega}$ is a (weakly/strongly) positive $\delta$-preform.
Let $\KC$ be a polyhedral complex of definition for $\tilde\alpha$  which means  
$$\tilde\alpha = \sum_{\Delta \in \KC^{l}} \tilde\alpha_\Delta \wedge \delta_\Delta$$
for  $\tilde\alpha_\Delta \in A^{p,p}(\tilde\Omega \cap \Delta)$. Since $\alpha$ is strongly positively representable, the superform $\tilde\alpha_\Delta$ is  strongly  positive on $\tilde\Omega \cap \Delta$. We may also assume that $\Ccal$ is a polyhedral complex of definition for $\tilde\beta$ and that  $\Trop(U)$ is a subcomplex of $\Ccal$. Note that the $\delta$-preform $\beta|_\Omega = \tilde\beta \wedge \delta_{\Trop(U)}$ has the polyhedral decomposition
$$\beta|_\Omega = \sum_{\Delta' \in \KC^{l'}} \beta_{\Delta'} \wedge \delta_{\Delta'} \in D(\Omega)$$
for (weakly/strongly) positive  $\beta_\Delta \in A^{p',p'}(\Omega \cap \Delta)$.   By the formula (2.12.3) in 
\cite{gubler-kuennemann} for the polyhedral representation of the product of $\delta$-preforms and using $\alpha \wedge \beta|_\Omega = \tilde\alpha \wedge (\tilde\beta \wedge \delta_{\Trop(U)})$, we have
$$\alpha \wedge \beta|_\Omega = \sum_{\tau \in \KC^{l+l'}} \sum_{\Delta,\Delta'} [N:N_\Delta + N_{\Delta'}]  \cdot \tilde\alpha_\Delta \wedge \beta_{\Delta'} \wedge \delta_{\tau} \in D(\Omega),$$
where $\Delta,\Delta'$ range over all pairs in $\KC^l \times \KC^{l'}$ with $\tau = \Delta \cap \Delta'$ and with $\Delta \cap (\Delta' + \ve v) \neq \emptyset$ for a fixed generic vector $v \in N_{U,\R}$ and for $\ve>0$ sufficiently small.  
  Proposition \ref{positivity properties of superforms} shows now that $\alpha \wedge \beta|_\Omega$ is a (weakly/strongly) positive $\delta$-preform. This proves (b).  
  
We can prove (c) in the same way as (b) if we observe 
Proposition \ref{positivity properties of superforms} (e). If $\alpha$ is a smooth form given on  $(V,\varphi_U)$ by a superform $\alpha_U$ on $\Omega$, then 
$$\alpha \wedge \beta|_\Omega = \sum_{\Delta' \in \KC} \alpha_U \wedge \beta_{\Delta'} \wedge \delta_{\Delta'} \in D(\Omega)$$
and the last claim follows again from Proposition \ref{positivity properties of superforms}. 
\end{proof}

\begin{definition} \label{positive delta-currents on W}
A  $\delta$-current $T \in E^{p,p}(W)$ is called {\it (weakly/strongly) positive} if $T$ is symmetric and if 
$\langle T, \beta \rangle \geq 0$
for all (strongly/weakly) positive $\delta$-forms $\beta \in B_c^{n-p,n-p}(W)$. 
\end{definition}

Recall from \cite[Prop. 6.6]{gubler-kuennemann} that we have a natural map $P^{p,q}(W) \to E^{p,q}(W)$, $\alpha \mapsto [\alpha]_E$ determined by
$
\langle[\alpha]_E,\beta\rangle=\int_W\alpha\wedge\beta
$ 
for all $\beta\in B_c^{n-p,n-q}(W)$, inducing  the map $A^{p,q}(W) \to D^{p,q}(W), \, \alpha \mapsto [\alpha]_D$.

\begin{cor} \label{delta forms vs delta currents}
Let $\alpha \in P^{p,p}(W)$ be a (weakly/strongly) positively representable
generalized $\delta$-form.  
Then $[\alpha]_E$ is a (weakly/strongly) positive $\delta$-current.
\end{cor}

\proof
We have to show $\int_W\alpha\wedge\beta\geq 0$
for every (strongly/weakly) positive $\beta\in B_c^{n-p,n-p}(W)$.
This follows from Proposition 
\ref{positivity representable products}(c).
\qed

\begin{rem} \label{positive currents on W}
A symmetric current $T \in D^{p,p}(W)$ is called {\it (weakly/strongly) 
positive} if {$\langle T,\alpha\rangle \geq 0$} for 
all (strongly/weakly) 
positive smooth forms $\alpha \in A^{n-p,n-p}(W)$ with compact support (see 
\cite[\S 5.3]{chambert-loir-ducros}). 
Since every smooth form is a $\delta$-form, Remark \ref{compare cld} shows that any 
(weakly/strongly) positive $\delta$-current induces a (weakly/strongly) 
positive current.  
\end{rem}

\begin{prop} \label{functorial current positive and positive delta-forms}
Let $\beta \in P^{p,p}(W)$. Then $\beta$ is a (weakly/strongly) positive generalized $\delta$-form if and only if  for every morphism $f:X' \rightarrow X$ of varieties the current $[f^*(\beta)]_D \in D^{p,p}(f^{-1}(W))$ is (weakly/strongly) positive.
\end{prop}

\begin{proof}
We may assume that $\beta$ has codimension $l$ and let $\alpha \in A_c^{n-p,n-p}(W)$. 
By  \cite[Prop. 5.7]{gubler-kuennemann}, there is a very affine open subset $U \subseteq X$  and an open subset $\Omega$ of $\Trop(U)$ such that $V := \trop_U^{-1}(\Omega)$ contains the support of $\alpha \wedge \beta \in P^{n,n}(W)$ and such that $\alpha$ (resp. $\beta$) is given on $V$ by $\alpha_U \in A^{n-p,n-p}(\Omega)$ (resp. $\beta_U \in P^{p,p}(V,\varphi_U)$).
Then $\alpha_U \wedge \beta_U$ has compact support in $\Omega$ \cite[Prop. 4.21]{gubler-kuennemann}.  
Since $\beta_U$ has codimension $l$, there is a polyhedral complex $\KC$ of definition for $\beta_U|_\Omega$ with polyhedral representation
$$\beta_U|_\Omega = \sum_{\Delta \in \KC^{l}} \beta_\Delta \wedge \delta_\Delta$$
for superforms $\beta_\Delta \in A^{p-l,p-l}(\Omega \cap \Delta)$. By \cite[Def. 5.8]{gubler-kuennemann},  we have
\begin{equation} \label{eqi}
\langle [\beta]_D, \alpha \rangle = \int_W \beta \wedge \alpha = \int_{|\Trop(U)|} \beta_U \wedge \alpha_U = \sum_{\Delta \in \KC^l} \int_\Delta \alpha_\Delta \wedge \beta_\Delta .
\end{equation}
Suppose that $\beta$ is a (weakly/strongly) positive generalized $\delta$-form and that  $\alpha \in A_c^{n-p,n-p}(W)$ is a (strongly/weakly) positive smooth form. It follows from Proposition \ref{positivity representable products} that $\beta \wedge \alpha$ is a positive generalized $\delta$-form of type $(n,n)$. By Proposition \ref{local positivity is fine}, the superform  
$\alpha_\Delta \wedge \beta_\Delta$  of type $(n-l,n-l)$ is positive on $\Delta \cap \Omega$.  We conclude that the integral in \eqref{eqi} is non-negative and hence $[\beta]_D$ is a positive current on $W$. If $f:X' \to X$ is a morphism of varieties, then $f^*(\beta)$ is a (weakly/strongly) positive generalized $\delta$-form on $f^{-1}(W)$ and hence $[f^*(\beta)]_D$ is a positive current on $f^{-1}(W)$.

Conversely, assume that $[f^*(\beta)]_D$ is a positive current on $f^{-1}(W)$ for all morphisms $f:X' \to X$. Using this functoriality, it is enough to show that $\beta_U|_\Omega$ is a (weakly/strongly) positive $\delta$-preform on $\Omega := \trop_U(V)$ for a tropical chart $(V,\varphi_U)$ where $\beta$ is given by $\beta_U \in P^{p,p}(V,\varphi_U)$. By assumption, the integral in \eqref{eqi} is  non-negative for $\alpha$ given by a  (strongly/weakly) positive superform $\alpha_U \in A_c^{n-p,n-p}(\Omega)$. It follows from Proposition \ref{positivity properties of superforms}(e)  
that $\beta_\Delta$ is (weakly/strongly) positive in $A^{n-l,n-l}(\Omega \cap \Delta)$ for every $\Delta \in \KC^l$. This proves that $\beta_U|_\Omega$  is a (weakly/strongly) positive $\delta$-preform. 
\end{proof}

\begin{rem} \label{positivity for smooth forms and its current}
The same argument shows for a smooth form $\beta \in A^{p,p}(W)$ that $\beta$ is (weakly/strongly) positive if and only if the associated current $[\beta]_D$ is (weakly/strongly) positive on $W$,  and furthermore this is equivalent that the $\delta$-current $[\beta]_E$ is (weakly/strongly) positive on $W$. 
\end{rem}

\begin{lem} \label{positivity of delta-forms and coverings}
Let $f:X' \rightarrow X$ be a generically finite dominant morphism of varieties over $K$ and let $W$ be open in $\Xan$. 
Then $\beta \in P^{p,p}(W)$ 
is (weakly/strongly) positive 
if and only if $f^*(\beta)$ is (weakly/strongly) positive on $f^{-1}(W)$.
\end{lem}

\proof 
This follows from Lemma \ref{positivity on tropical charts and coverings}. 
\qed

\begin{lem}\label{positive and negative generalized delta-form}
Let $W$ be an open subset of $X^\an$ and let  
$\alpha\in P^{p,p}(W)$  such
that $\alpha$ and $-\alpha$ are weakly positive.
Then $\alpha=0$.
\end{lem}

\proof
We choose a tropical chart $(V,\varphi_U)$ of $X^\an$ with
$V\subseteq W$ such that $\alpha$ is induced by $\alpha_U$
in $P^{p,p}(V,\varphi_U)$.
We get from Proposition \ref{local positivity is fine} 
that $\alpha_U$ and $-\alpha_U$ are positive in the sense of 
\ref{positivity on tropical charts}.
We conclude from Remark \ref{posnegtropchart} that $\alpha_U=0$.
It follows that $\alpha$ vanishes as well.  
\qed

\section{Plurisubharmonic functions and  metrics} \label{pluri}

Let $X$ be an $n$-dimensional variety over $K$.
We recall first some definitions from \cite{chambert-loir-ducros}. 

\begin{definition} \label{psh functions}
A continuous real function $f$ on an open subset $W$ of $\Xan$ is called 
{\it plurisubharmonic} or \emph{psh} if $d'd''[f]_D$ is a positive current 
in $D^{1,1}(W)$. 
This means that $d'd''[f]_D$ has to be non-negative on positive forms in $A_c^{n-1,n-1}(W)$. 
\end{definition}

\begin{rem} \label{smooth psh}
There is an elementary way to describe 
when a smooth function $f$ is psh.  
Locally, there is a tropical chart $(V,\varphi_U)$ and a smooth function 
$\phi$ on $\Omega = \trop_U(V)$ with $f= \phi \circ \trop$. 
Then it follows from \cite[Lemma 5.5.3]{chambert-loir-ducros} that $f$ is psh 
on $V$ if and only if the restriction of $\phi$ to $\Delta$ is convex for 
any polyhedron $\Delta \subseteq \Omega$. 
\end{rem}

Next we show that Remark \ref{smooth psh} doesn't hold
without the smoothness assumption on $f$ (see \S \ref{Piecewise smooth metrics} for a discussion of piecewise smooth functions).

\begin{ex} \label{counterexample for convexity}
Let us consider the line $x_1+x_2=1$ in $\G_m^2$. 
Then $\Trop(X)$ is the union of the three half lines 
$\{u \in \R_+^2 \mid u_1=0\}$, $\{u \in \R_+^2 \mid u_2=0\}$ and $\{u \in \R_-^2 \mid u_1=u_2\}$, all equipped with weight $1$. 
Let $\phi$ be the conic function determined by $\phi(1,0)=a$, $\phi(0,1)=b$ and 
$\phi(-1,-1)=c$. As above, we set $f:=\phi \circ \trop$. 
Then the restriction of $\phi$ to any polyhedron $\Delta \subseteq \Trop(X)$ 
is linear and hence convex. 
However, if $\eta$ is a non-negative compactly supported smooth function 
on $\Trop(X)$ with $\eta(0,0)=1$, then we have 
\[
\langle d'd''[f]_E, \trop^*(\eta) \rangle = \langle  [\phi], d'd''\eta \rangle = (a+b+c) \eta(0,0)= a+b+c
\]
which gives a counterexample to $[f]$ psh if and only if $a+b+c <0$.
\end{ex}

\begin{art} \label{functorial psh}
We give some variants of defining psh functions. 
We always consider a continuous function $f$ on an open subset $W$ of $\Xan$. 
\begin{itemize}
\item[(a)] 
$f$ is {\it $\delta$-psh} if $d'd''[f]_E$ is a positive 
$\delta$-current in $E^{1,1}(W)$ as defined in \ref{positive delta-currents on W} using the $\delta$-current $[f]_E$ from \cite[Prop. 6.16]{gubler-kuennemann}. 
\item[(b)] $f$ is {\it functorial psh} if $f \circ \varphi$ is psh on 
$\varphi^{-1}(W)$ for all morphisms $\varphi:X' \rightarrow X$.  
\item[(c)] $f$ is {\it functorial $\delta$-psh} if $f \circ \varphi$ is 
$\delta$-psh on $\varphi^{-1}(W)$ for all morphisms $\varphi:X' \rightarrow X$.
\end{itemize}
Clearly (c) yield (a) and (b).
either (a) or (b) yield that the function $f$ is psh.
\end{art}

\begin{art} \label{psh-metrics}
In the following, $L$ is a line bundle on $X$. 
Let $\metr$ be a continuous metric on $L^\an$ over an open subset $W$
of $X^\an$.
Following \cite[6.3.1]{chambert-loir-ducros} the metric $\metr$ is called 
{\it psh} if $-\log \| s \|$ is a psh-function on 
$\Uan\cap W$ for any frame $s$ of $L$ over any open subset $U$. 
 {Recall that a frame of a line bundle over $U$ is
{by definition} a nowhere vanishing section over $U$.} 
Note that this is equivalent to say that $[c_1(L,\metr)]_D$ is a 
positive current on $W$. 
Similarly as above we say that
\begin{itemize}
\item[(a)] 
$\metr$ is {\it $\delta$-psh} if $-\log \| s \|$ is a 
$\delta$-psh-function on $\Uan\cap W$ for any frame $s$ of $L$ over 
any open subset $U$ of $X$. 
\item[(b)] 
$\metr$ is {\it functorial psh} if $\varphi^* \metr$ is psh on 
$(\varphi^\an)^{-1}(W)$ for all morphisms $\varphi:X' \rightarrow X$.  
\item[(c)] 
$\metr$ is {\it functorial $\delta$-psh} if 
$\varphi^*\metr$ is $\delta$-psh on $(\varphi^\an)^{-1}(W)$ for all morphisms 
$\varphi:X' \rightarrow X$.
\end{itemize}
Clearly (c) implies (a) and (b).
Furthermore either (a) or (b) implies that the metric $\metr$ is psh. Note also that $\metr$ is $\delta$-psh if and only if 
the first Chern $\delta$-current $[c_1(L,\metr)]_E$ (defined in \cite[7.7]{gubler-kuennemann}) is a 
positive $\delta$-current on $W$. 
\end{art}

\begin{prop} \label{push-forward of function}
Let $\varphi:X' \to X$ be a surjective proper morphism of $n$-dimensional varieties over $K$ and let $f$ be a continuous function on an open subset $W$ of $\Xan$. Then we have the projection formula 
$$\varphi_*[\varphi^*f]_E = \deg(\varphi) [f]_E,$$
where $[\varphi^*f]_E$ is the $\delta$-current on $\varphi^{-1}(W)$ induced by $f \circ \varphi$.  
\end{prop}

\begin{proof} 
Let $\alpha \in B_c^{n,n}(W)$ and let $\mu_\alpha$ be the associated Radon measure on $W$ (see \cite[Cor. 6.15]{gubler-kuennemann}). It follows from the projection formula in \cite[Prop. 5.9]{gubler-kuennemann} that 
\begin{equation} \label{projection formula}
 \deg(\varphi)\int_W g \alpha = \int_{\varphi^{-1}(W)} \varphi^*(g) \varphi^*(\alpha)
\end{equation}
for all smooth functions $g$ on $W$. Since the smooth functions with compact support in $W$ are dense in the space of continuous functions with compact support in $W$ equipped with the supremum norm \cite[Prop. 3.3.5]{chambert-loir-ducros}, we conclude from \eqref{projection formula} that $\deg(\varphi)\mu_\alpha = \varphi_*(\mu_{\varphi^*\alpha})$ as an identity of Radon measures. In particular, we get 
$$\deg(\varphi) \langle [f]_E, \alpha \rangle = \deg(\varphi) \int_W f \,d\mu_\alpha = \int_W f \, d\varphi_*(\mu_{\varphi^*\alpha}) = 
\langle [\varphi^*f]_E , \varphi^*\alpha \rangle $$
proving the claim. 
\end{proof}

In the following, we consider a  continuous metric $\metr$ on $\Lan$ over the open subset $W$ of $\Xan$ as before. 

\begin{cor} \label{projection formula for chern-current}
Let $\varphi:X' \to X$ be a surjective proper morphism of $n$-dimensional varieties over $K$. Then we have 
$$\varphi_*([c_1(\varphi^*(L)|_{\varphi^{-1}(W)},
\varphi^*\metr)]_E) = \deg(\varphi) [c_1(L|_W,\metr)]_E$$
as an identity of $\delta$-currents on $W$.
\end{cor}

\begin{proof}
For a $\delta$-metric (i.e. $c_1(L,\metr)$ is a $\delta$-form), this identity follows directly from the projection formula in \cite[Prop. 5.9]{gubler-kuennemann}. It is clear that $L$ has a $\delta$-metric $\metr_0$ as we may choose a smooth metric \cite[Prop. 6.2.6]{chambert-loir-ducros} or a formal metric of a compactification of $X$ (see Section \ref{forma}) using \cite[Rem. 9.16]{gubler-kuennemann}. Using that the claim holds for $\metr_0$ and linearity, it remains to prove the corollary in the special case $L=O_X$ with a metric induced by a continuous function $f$ on $W$, i.e. $f=-\log \|1\|$. For $\alpha \in B_c^{n,n}(W)$, we have 
$$\langle \varphi_*([c_1(\varphi^*(L)|_{\varphi^{-1}(W)},
\varphi^*\metr)]_E) , \alpha \rangle = \langle d'd''[\varphi^*f]_E , \varphi^* \alpha \rangle = 
\langle \varphi_*[\varphi^*f]_E, d'd'' \alpha \rangle$$
and
$$\langle [c_1(L|_W,\metr)]_E , \alpha \rangle = \langle d'd''[f]_E , \alpha \rangle = \langle [f]_E, d'd'' \alpha \rangle.$$
We conclude that the claim follows from Proposition \ref{push-forward of function}.
\end{proof}

\begin{cor} \label{descent of psh}
Let $\varphi:X' \to X$ be a surjective proper morphism of $n$-dimensional varieties over $K$ and let $\metr$ be a  continuous metric  on $\Lan$ over the open subset $W$ of $\Xan$. If $\varphi^*\metr$ is psh (resp. $\delta$-psh, resp. functorial psh, resp. functorial $\delta$-psh) over $\varphi^{-1}(W)$, then $\metr$ is psh (resp. $\delta$-psh, resp. functorial psh, resp. functorial $\delta$-psh) over $W$. 
\end{cor}

\begin{proof} If $\varphi^*\metr$ is $\delta$-psh over $\varphi^{-1}(W)$, then $\metr$ is $\delta$-psh by Corollary \ref{projection formula for chern-current}. Indeed,  the proper push-forward of a positive $\delta$-current is positive since positivity of $\delta$-forms is closed under pull-back. All these facts for $\delta$-currents yield immediately the corresponding facts for currents and so the same argument works for psh. Using a suitable cartesian diagram, the remaining two claims involving functoriality follow easily.
 \end{proof}

\section{Lifting varieties}\label{lifti}

Let $(F,|\phantom{a}|)$ be a field with a non-archimedean absolute value.
Let $F^\circ$, $F^{\circ\circ}$, and $\tilde F=F^\circ/F^{\circ\circ}$
denote the valuation ring, its maximal ideal, and the
corresponding residue class field. 

{The following theorem enables us to lift closed subsets from the special 
fibre of a $F^\circ$-model to the generic fibre. Amaury Thuillier has 
told the authors that he has found a similar argument.}

\begin{thm}\label{liftvarieties}
Let $\Xcal$ denote a flat scheme of finite type over 
${\rm Spec}\,F^\circ$ with generic fibre $X=\Xcal_\eta$
and special fibre $\Xcal_s$.
Let $V$ be an irreducible closed subset of $\Xcal_s$.
Then there exists an integral closed subscheme $Y$ of $X$ such
that $V$ is an irreducible component of the special 
fibre $(\overline{Y})_s$ of the schematic closure
$\overline Y$ of $Y$ in $\Xcal$  {and such that $\dim(Y)=\dim(V)$}. 
\end{thm}

\proof
We may assume without loss of generality 
that the absolute value on $F$ is non-trivial and
that $\Xcal={\rm Spec}\,A$ is an affine scheme.
We consider $V$ as an integral closed subscheme of $\KX_s$ .
Let $r$ denote its dimension.
We choose a closed embedding
\[
\Xcal={\rm Spec}\,A\hookrightarrow \A_{F^\circ}^N.
\]
As in the proof of Noether normalization, we can choose 
a generic projection $\A_{\tilde F}^N\to \A_{\tilde F}^r$
such that the induced morphism 
\[
\psi:V\hookrightarrow \Xcal_s\to \A_{\tilde F}^r
\] 
is finite and surjective.
The morphism $\Xcal_s\to \A_{\tilde F}^r$ clearly lifts to a morphism
\[
\varphi:\Xcal\longrightarrow \A_{F^\circ}^r.
\]
We equip the function field $L=F(x_1,\ldots,x_r)$ of 
 {$\A_{F^\circ}^r=\Spec (F^\circ[x_1,\ldots,x_r])$} 
with the Gauss norm.
Base change to the valuation ring $L^\circ$ of $L$ yields a cartesian diagram
\[
\xymatrix{
\Xcal' \ar[d]^p \ar[r] &{\rm Spec}\,L^\circ \ar[d]\\
\Xcal\ar[r]^\varphi&\A_{F^\circ}^r}.
\]
and corresponding cartesian diagrams for the generic and the special fibre
\[
\xymatrix{
X'\ar[d]\ar[r]&{\rm Spec}\,L\ar[d]&
(\Xcal')_s\ar[d]\ar[r]&{\rm Spec}\,(\tilde L)\ar[d]\\
X\ar[r]&\A_F^r&\Xcal_s\ar[r]&\A_{\tilde F}^r.}
\]
The residue class field $\tilde L=\tilde F(x_1,\ldots,x_r)$ of $L$
is canonically isomorphic to the residue class field 
$\kappa(\eta)$ of the generic
point $\eta$ of $\A_{\tilde F}^r$.
Base change yields the canonical diagram
\[
\xymatrix{
V'=\psi^{-1}(\{\eta\})\ar[d]\ar[r]&
(\Xcal')_s\ar[d]\ar[r]&{\rm Spec}\,(\tilde L)\ar[d]_\eta\\
V\ar[r]&\Xcal_s\ar[r]&\A_{\tilde F}^r}
\]
which shows that $V'$ is finite over $\tilde L$.
The generic point $\eta_V$ of $V$ maps to $\eta$
and determines the generic point $\eta_{V'}$
of the fibre $\psi^{-1}(\{\eta\})$. 
Then $\eta_{V'}$ is a closed point on $(\Xcal')_s$ as $V'$ is finite over $\tilde L$.
We have $\KX'={\rm Spec}\,A'$ for $A'=A\otimes_{F^\circ[x_1,\ldots, x_r]}L^\circ$. Any finitely generated ideal in a valuation ring $R$ is principal and hence an $R$-module is flat if and only if it is torsion free. 
It follows that $L^\circ$ is a flat $F^\circ[x_1,\ldots, x_r]$-module and hence $A'$ is a flat $A$-module. Since $A$ is a flat $F^\circ$-algebra, we deduce that $A'$ is a flat $F^\circ$-module. Then $A'$ is also torsion free as an $L^\circ$-module and hence flat over $L^\circ$. 
We choose a non-zero $\rho\in L^{\circ\circ}$ and define 
\[
\widehat{\Xcal'}={\rm Spf}\,\lim_{\genfrac{}{}{0pt}{}{\longleftarrow}{k}}A'/\rho^kA'.
\]
as the $\rho$-adic completion of $\Xcal'$. 
Let $\hat L$ denote the completion of $L$.
Then $\widehat{\Xcal'}$ is an admissible formal $\hat L^{\circ}$-scheme 
in the sense of \cite[Sect. 7.4]{bosch-lectures-2015} as shown in \cite[\S 6.1]{ullrich}.
The generic fibre $(X')^\circ$ of $\widehat{\Xcal'}$ is the 
affinoid Berkovich analytic space associated with the strict
affinoid $\hat L$-algebra
\[
\KA=\hat L\otimes_{\hat L^\circ}\lim_{\genfrac{}{}{0pt}{}{
\longleftarrow}{k}}A'/\rho^kA'.
\]
It coincides with the affinoid domain in $(X')^\an$ given by all points
in $(X')^\an$ whose reduction in 
$(\widehat{\Xcal'})_s=(\Xcal')_s$ is defined 
\cite[\S 4]{gubler-guide}.
The reduction map
\[
\pi\,:(X')^\circ=\Mcal(\KA)\longrightarrow (\widehat{\Xcal'})_s=\Xcal'_s
\]
is surjective 
and 
anti-continuous in the sense that inverse images of closed subsets
are open.
{This follows from  \cite[Prop. 2.4.4, Cor. 2.4.2]{berkovich-book} as explained in \cite[\S 2]{grw2}.} 
The closed point $\eta_{V'}$ in $(\Xcal')_s$ yields the open subset
$\pi^{-1}(\{\eta_{V'}\})$ of $\Mcal(\KA)$.
We also get that $\pi^{-1}(\{\eta_{V'}\})$ is contained in the relative 
interior ${\rm Int}\,(\Mcal(\KA) / \Mcal(L))$
by \cite[Lemme 6.5.1]{chambert-loir-ducros}.
{By \cite[Thm. 3.4.1]{berkovich-book}, $(X')^\an$ is a closed analytic space which means boundaryless and hence \cite[Prop. 3.1.3(ii)]{berkovich-book} yields}
\[
{\rm Int}\,(\Mcal(\KA) / (X')^\an)={\rm Int}\,(\Mcal(\KA) / \Mcal(L)).
\]
The set on the lefthand side coincides by 
\cite[Prop. 3.1.3(i)]{berkovich-book} with the topological interior of 
$(X')^\circ$ in $(X')^\an$.
It follows that $\pi^{-1}(\{\eta_{V'}\})$ is open in $(X')^\an$. 
Each closed point $\xi'$ of $X'$ determines a closed point
in $(X')^\an$ (which we denote again by $\xi'$) and  these 
points form a dense subset of $(X')^\an$ by \cite[2.6]{gubler-guide}. 
Hence there exists a closed point $\xi'\in X'$ such that
$\pi\,(\xi')=\eta_{V'}$.
We get $\eta_{V'} \in \overline{\{\xi'\}}$ where the  closure
is taken in $X'$ \cite[4.8]{gubler-guide}.
Let $p:\Xcal'\to \Xcal$ denote the base change morphism.
Let $Y$ be the Zariski-closure $\overline{\{p(\xi')\}}$ of $p(\xi')$ 
in $X$.
It follows from our construction that $p(\eta_{V'})=\eta_V$.
Let $\overline Y$ denote the closure of $p(\xi')$ in $\Xcal$.
From $\eta_{V'}\in \overline{\{\xi'\}}$ we get $\eta_V\in \overline Y$ and 
hence $V\subseteq \overline Y$.
The flatness of $\overline Y$ over $\Spec\, F^\circ$ yields
$\dim Y=\dim (\overline Y)_s$.
It remains to show $\dim Y\leq \dim V$.
We get ${\rm trdeg}\,\bigl(\kappa(\xi')/L\bigr)=0$ as
$\xi'$ is a closed point of $X'$.
Hence
$$\dim Y
\leq{\rm trdeg}\,\bigl(\kappa(p(\xi'))/F\bigr)
\leq{\rm trdeg}\,\bigl(\kappa(\xi')/L\bigr)+{\rm trdeg}\,(L/F)
=\dim V$$
yields our claim.
\qed

\section{Formal metrics}\label{forma}

Recall that $K$ denotes always an algebraically closed field endowed with a complete non-trivial non-archimedean absolute value $|\phantom{a}|$. 
In this section, we gather various properties of formal metrics on line bundles of strictly $K$-analytic spaces. Such metrics play an important role in Arakelov geometry as we can use the underlying model for intersection theory (see \cite{gubler-crelle}).

\begin{art} \label{models}
Let $X$ be a compact reduced strictly $K$-analytic space in the sense of \cite{berkovich-ihes} (resp. a proper algebraic variety). 
A {\it formal model} (resp. {\it algebraic model}) of $X$ over $\kcirc$ is an admissible formal scheme (resp. a flat proper integral scheme) $\Xcal$ over $\kcirc$ with a fixed isomorphism from the generic fibre $\Xcal_\eta$ onto $X$. 
 {We assume reduced in the analytic case and integral in the algebraic case for simplicity as this fits to the general setup of this paper. We refer to \cite{gubler-martin-2016} for generalizations.}
We recall that an admissible formal scheme $\Xcal$ is a flat formal scheme over $\kcirc$ which is locally of topologically finite type. We will always identify $\Xcal_\eta$ with $X$ along the fixed isomorphism. 
\end{art}

\begin{rem} \label{Raynaud's theorem}
By \cite[Thm. 1.6.1]{berkovich-ihes}, the category of compact strictly $K$-analytic spaces  is equivalent to the category of quasicompact and quasiseparated rigid $K$-analytic varieties. This allows us to use {\it Raynaud's theorem}, proved by Bosch and L\"utkebohmert in \cite[Thm. 4.1]{bosch-luetkebohmert-1},  which shows that the category of quasicompact admissible formal schemes over $\kcirc$ localized in the class of admissible formal blowing ups is equivalent to the category of quasicompact and quasiseparated rigid $K$-analytic varieties. Note that this holds more generally with quasiparacompact replacing quasicompact \cite[Thm. 8.4.3]{bosch-lectures-2015}, but we don't need to work in such generality.

In the algebraic setting, {\it Nagata's compactification theorem} replaces Raynaud's theorem from above. It shows that for an algebraic variety $X$ over $K$ 
{there is a} flat proper variety $\Xcal$ over $\kcirc$ such that $X$ is an open dense subset of $\Xcal_\eta$.
This was proved by Nagata in the noetherian case (see \cite{nagata-1962,nagata-1963}) 
{and} generalizes to varieties over valuation rings by noetherian approximation. In particular, it is no restriction of generality working with proper schemes over $K$ in the algebraic case.
\end{rem}

\begin{art} \label{models and metrics}
We first recall that a line bundle on a strictly $K$-analytic space is a locally free sheaf of rank $1$ on the $\rm G$-topology. If the $K$-analytic space is good, then it is equivalent to have a locally free sheaf on the Berkovich topology  \cite[Prop. 1.3.4]{berkovich-ihes}. In this paper, we consider exclusively the $\rm G$-topology on strictly $K$-analytic spaces induced by the strictly affinoid subdomains. This $\rm G$-topology allows us to use results from rigid geometry (see \cite[\S 1.6]{berkovich-ihes}).

Let $L$ be a line bundle on the compact reduced strictly $K$-analytic space (resp. proper algebraic variety) $X$. A {\it formal model} (resp. {\it algebraic model}) of $(X,L)$ over $\kcirc$ is a pair $(\Xcal,\Lcal)$ where $\Xcal$ is a formal model (resp. algebraic model) of $X$ over $\kcirc$ and where $\Lcal$ is a line bundle on $\Xcal$ with a fixed isomorphism $\Lcal|_X \cong L$  which we will use again for identification.

A formal (resp. algebraic) $\kcirc$-model $(\Xcal,\Lcal)$ of $(X,L)$ gives rise to an {\it associated formal  metric} (resp. {\it associated algebraic metric}) $\metr_\Lcal$ on $L$ in the following way: for $x \in X$, let us choose a trivialization $\Ucal$ of $\Lcal$ in 
 {a neighbourhood of} 
the reduction 
$\pi(x)$ of $x$. The induced isomorphism $\Lcal(\Ucal) \cong \Ocal(\Ucal)$ allows to identify a local section $s$ with a regular function $\gamma$ and then we define 
\begin{equation} \label{definition formal metric}
\|s(x)\|_\Lcal = |\gamma(x) |.
\end{equation}
This is independent of the choice of the trivialization as a change involves multiplication with an invertible function in $\Ocal_{\Xcal,\pi(x)}$.
\end{art}

\begin{definition} \label{formal metrics}
A metric $\metr$ on $L$ is called a {\it formal metric} (resp. {\it algebraic metric}) if there is a formal (resp. algebraic) $\kcirc$-model $(\Xcal,\Lcal)$ of $(X,L)$ such that $\metr= \metr_\Lcal$. More generally, a {\it $\Q$-formal (resp. $\Q$-algebraic) metric}  on $L$ is a metric $\metr$  on $L$ such that there is a non-zero $n \in \N$ with $\metr^{\otimes n}$ a formal (resp. algebraic) metric 
 {of $L^{\otimes n}$}. 
\end{definition}

\begin{prop} \label{properties of formal metrics}
Let $L$ be a line bundle on the compact reduced strictly $K$-analytic space $X$. Then the following properties hold:
\begin{itemize}
\item[(a)] $\Q$-formal metrics on $L$ are continuous on $X$.
\item[(b)] $\Q$-formal metrics on $L$ are dense in the space of continuous metrics on $L$ with respect to  uniform convergence. 
\item[(c)] $L$ has a formal metric.
\item[(d)] The isometry classes of formally (resp. $\Q$-formally) metrized line bundles over $X$ form an abelian group. 
\item[(e)] The pull-back of a formal (resp. a $\Q$-formal) metric on $L$ with respect to a morphism $f:X' \to X$ of compact reduced strictly analytic spaces is  a formal (resp. a $\Q$-formal) metric on $f^*(L)$.
\item[(f)] The maximum and the minimum of two formal metrics on $L$ are again formal metrics on $L$. 
\end{itemize} 
\end{prop}

\begin{proof}
Continuity in (a) means that $\|s\|$ is continuous for any local section $s$. This follows easily from \eqref{definition formal metric}. 

For (b), we note that the quotient of two metrics on $L$ gives rise to a metric on $O_X$ and evaluation at the constant section $1$ gives rise to a continuous function $f$ on $X$. Then we use the maximum norm of $|\log(f)|$  to measure the distance of the two metrics and claim (b) follows from \cite[Thm. 7.12]{gubler-crelle}.

Property (c) is shown in \cite[Lemma 7.6]{gubler-crelle} and 
 (d) follows  
from \eqref{definition formal metric}.

To prove (e), let $\metr_\Lcal$ be the formal metric on $L$ associated to the $\kcirc$-model $(\Xcal, \Lcal)$ of $(X,L)$. We use Raynaud's theorem from \ref{Raynaud's theorem} which shows the existence of a $\kcirc$-model $\Xcal'$ of $X'$ and of a morphism $\varphi:\Xcal'\rightarrow \Xcal$ with generic fibre $f:X' \to X$. Then (e) follows from 
\begin{equation} \label{line bundle and pull-back metric}
 \varphi^*(\metr_\Lcal)=\metr_{\varphi^*(\Lcal)}.
\end{equation}
Finally, (f) is proven in \cite[Lemma 7.8]{gubler-crelle}.
\end{proof}

\begin{rem} \label{properties of algebraic metrics}
For a proper variety $X$ over $K$, we have similar properties as in Proposition \ref{properties of formal metrics} formulated on $\Xan$. Most of them can be proved in the same way replacing Raynaud's theorem by Nagata's compactification theorem 
 {(use Vojta's version in \cite[Thm. 5.7]{vojta-2007}) for line bundles).} 
However, they also can be deduced from the fact that algebraic and formal metrics are the same on a line bundle over a proper variety \cite[Prop. 8.13]{gubler-kuennemann}.
\end{rem}

\begin{rem} \label{reduced special fibre}
Let $\Xcal$ be a formal $\kcirc$-model of the compact reduced strictly $K$-analytic space $X$. Then there is a canonical $\kcirc$-model $\Xcal'$ of $X$ over $\Xcal$ and a canonical morphism $\iota:\Xcal' \to \Xcal$ extending the identity such that $\Xcal'$ has reduced special fibre. It is obtained by covering  {the admissible formal scheme} $\Xcal$ by formal affine open subschemes $\Ucal = \Spf(A)$, noting that $\Acal := A \otimes_{\kcirc} K$ is a reduced strictly affinoid algebra and then gluing the admissible formal affine schemes $\Spf(\Acal^\circ)$ over $\kcirc$. 
 {Note that admissibility of the latter follows from our assumptions that $X$ is reduced and $K$ is algebraically closed {\cite[\S 6.4.3]{bosch-guentzer-remmert}}.} 
It is a standard fact that $\iota$ induces a finite surjective morphism between the special fibres (see \cite[1.10, Prop. 1.11]{gubler-crelle}).
\end{rem}

\begin{lem} \label{formal metric and model}
Let $\metr$ be a formal metric on the line bundle $L$ of the compact reduced strictly $K$-analytic space $X$. Then there is a model $(\Xcal,\Lcal)$ of $(X,L)$ over $\kcirc$ with $\Xcal_s$ reduced and with $\metr = \metr_\Lcal$. Moreover, for such a model $\Xcal$, the sheaf $\Lcal$ is canonically isomorphic to the sheaf
$$\Ucal \mapsto  \{s \in L(U) \mid \|s(x)\| \leq 1\}, $$
where $\Ucal$ ranges over all open subsets of $\Xcal$ and  $U$ is the generic fibre of $\Ucal$.
\end{lem}

\begin{proof}
The first claim follows 
{from} {Remark \ref{reduced special fibre}.}  
The second claim follows from \cite[Prop. 7.5]{gubler-crelle}.
\end{proof}

\begin{art} \label{non-reduced analytic space}
Let $X$ be a compact strictly $K$-analytic space which is not necessarily reduced and let $L$ be a line bundle on $X$. Then it is better to work with piecewise linear (resp. piecewise $\Q$-linear) metrics on $L$ (see \cite[\S 7]{gubler-crelle} for  details). Here, a metric $\metr$ on $L$ is called {\it piecewise linear}  if there is a $\rm G$-covering of $X$ which has frames of norm identically one, and {\it piecewise $\Q$-linear} if there is a non-zero $n \in  \N$ such that $\metr^{\otimes n}$ is a piecewise linear metric of $L^{\otimes n}$. 
The properties of Proposition \ref{properties of formal metrics} hold also for piecewise linear (resp. piecewise $\Q$-linear) metrics. If $X$ is reduced, then a piecewise linear (resp. piecewise $\Q$-linear) metric is the same as a formal (resp. $\Q$-formal) metric. In general, $X$ and the analytic space $X_{\rm red}$ with the induced reduced structure have the same $\rm G$-topology (see \cite[p. 389]{bosch-guentzer-remmert}). We conclude that pull-back gives a bijective correspondence between  piecewise linear (resp. piecewise $\Q$-linear) metrics  on $L$ and formal (resp. $\Q$-formal) metrics on  $L|_{X_{\rm red}}$. 
\end{art} 

\begin{prop} \label{local definition}
Let $X$ be a compact  strictly $K$-analytic space with a line bundle $L$. Then the definition of a piecewise linear metric is $\rm G$-local. 
\end{prop}

\begin{proof}
We have to show that if there is a $\rm G$-covering $(V_i)_{i\in I}$ of $X$ such that the restriction of the metric $\metr$ on $L$ to $V_i$ is piecewise linear for all $i \in I$, then $\metr$ is a piecewise linear metric. Passing to a refinement, we may assume that every $V_i$ has a frame of norm identically $1$ and hence $\metr$ is piecewise linear. 
\end{proof}

\begin{prop} \label{extension of PL-metrics}
Let $L$ be a line bundle on the compact strictly $K$-analytic space $X$ and let $V$ be a compact  {strictly $K$-analytic} domain in $X$. 
\begin{itemize}
\item[(i)] Let $\metr_V$ be a piecewise linear metric on $L|_V$. Then there is a piecewise linear metric $\metr$ on $L$ which extends $\metr_V$.
\item[(ii)]  {Let $(\Vcal,\Lcal_V)$ be a formal model of $(V,L|_V)$. Then there is a 
formal model $(\Xcal,\Mcal')$ of $(X,L)$ and a formal open subset $\Vcal'$ of $\Xcal$ which is a formal model of $V$ such that there is a 
 morphism $\psi:\Vcal'\to \Vcal$  with $\psi|_V=\id_V$ and with $\psi^*(\Lcal_V)=\Mcal'|_{\Vcal'}$.} 
\end{itemize}
\end{prop}

\begin{proof}
By the final remark in \ref{non-reduced analytic space}, we may assume that $X$ is reduced and  {that $\metr_V$ is a formal metric.}  
Then there is a formal model $(\Vcal,\Lcal_\Vcal)$ of $(V,L)$  {with associated formal metric $\metr_V$ and it is enough to prove (ii).}  
By Proposition \ref{properties of formal metrics}(c), there is a formal model $(\Xcal,\Lcal)$ of $(X,L)$. Modifiying $\Xcal$ by an admissible blowing up and replacing $\Vcal$ by a dominating formal model of $V$, Raynaud's theorem gives a morphism $\Vcal \to \Xcal$ extending the $\rm G$-open immersion $V \to X$. By \cite[Cor. 5.4]{bosch-luetkebohmert-2}, we may even assume that $\Vcal \to \Xcal$ is an open immersion. By 
 {Remark \ref{reduced special fibre},} we may assume that $\Xcal_s$ and hence $\Vcal_s$ are reduced.

We compare the sheaves $\Lcal_\Vcal$ and $\Lcal$ using the generic fibre $L$ as a reference, i.e. for any formal open subset $\Ucal$ of $\Vcal$ (resp. $\Xcal$) with generic fibre $U$, we view $\Lcal_\Vcal(\Ucal)$  (resp. $\Lcal(\Ucal)$) as a subset of $L(U)$. Using compactness of $V$ and replacing $\metr_\Lcal$ by a suitable  
 {multiple with a small number in $|K^\times|$ (which is dense in $\R_+$ by our assumptions on $K$),} we may assume that $\metr_\Lcal \leq \metr_V$ on $L|_V$. By Lemma \ref{formal metric and model}, we deduce that $\Lcal_\Vcal$ is a coherent submodule of $\Lcal$. Similarly as in the proof of \cite[Lemma 5.7]{bosch-luetkebohmert-1}, we can extend $\Lcal_\Vcal$ to a coherent submodule $\Ncal$ of $\Lcal$. Since $\Lcal|_\Vcal$ is coherent and $V$ is compact, there is a sufficiently small $\pi \in \kcirc \setminus \{0\}$ with $\pi \Lcal|_\Vcal$ a submodule of $\Lcal_\Vcal$. Then the generic fibre of the coherent submodule $\Mcal:=\Ncal +\pi \Lcal$ of $\Lcal$ is $L$ and $\Mcal$ agrees with $\Lcal_\Vcal$ on $\Vcal$. Using the flattening techniques from \cite[Thm. 4.1, Prop. 4.2]{bosch-luetkebohmert-2}, there is an admissible formal blowing up $\Xcal'$ of $\Xcal$ with center outside $\Vcal$ such that the strict transform $\Mcal'$ of $\Mcal$ is flat over $\Xcal'$. We conclude that $\Mcal'$ is a line bundle on $\Xcal'$ which agrees with $\Lcal_\Vcal$ over $\Vcal$. Since $\Mcal'|_X=L$, 
 {this proves (ii).}
\end{proof}

The following result shows that local analytic considerations for formal metrics on an algebraic variety can be always done with  algebraic metrics.

\begin{cor} \label{algebraic vs locally formal}
Let $L$ be a  line bundle on a proper variety $X$ over $K$. Suppose that $\metr$ is a formal metric on $L^\an|_V$ for a compact strictly $K$-analytic domain $V$ of $\Xan$. Then there is an algebraic metric $\metr'$ on $L$ which agrees with $\metr$ over $V$. 
\end{cor}

\begin{proof}
By Proposition \ref{extension of PL-metrics}, we can extend  $\metr$  to a formal metric $\metr'$ on $L$. By Remark \ref{properties of algebraic metrics}, this is an algebraic metric.
\end{proof}

\section{Semipositive piecewise linear metrics}\label{semip}

In this section, we start with a line bundle $L$ on a strictly $K$-analytic space $X$ over $K$ endowed with a piecewise linear metric $\metr$. 
We will introduce semipositive piecewise linear metrics 
from a point of view which is local on $X$. 
Assuming that $X$ is the analytification of 
a proper algebraic variety, 
we have seen in the previous section that piecewise linear, 
formal and algebraic metrics are the same. 
In this case, we will show that a formal metric is semipositive 
in all points of $X$ if and only if an associated model is vertically nef,  which is Zhang's definition used in arithmetic intersection theory. Then we show that semipositivity for formal metrics agrees with various other positivity notions introduced before. 

\begin{art} \label{extended definition of piecewise linear}
First, we generalize the definitions from Section \ref{forma} to our setting. 
Let $L$ be a line bundle on the strictly $K$-analytic space $X$ over $K$ which means that $L$ is a locally free sheaf of rank $1$ on the $\rm G$-topology of $X$. 
We say that a metric $\metr$ on $L$ is {\it piecewise linear} if there is $\rm G$-open covering which has frames of norm identically one. It is easy to see that piecewise linearity is closed with respect to the following operations: tensor product of metrics, passing to the dual metric and pull-back of metrics.
\end{art}

We have seen in \ref{non-reduced analytic space} that on a reduced compact strictly $K$-analytic space, a metric is piecewise linear if and only if it is formal. 
This will be used for the following local definition of semipositivity  which was suggested to us by Tony Yue Yu.

\begin{definition} \label{semipositive formal metric in x}
Let $\metr$ be a piecewise linear metric on the line bundle $L$ of the strictly $K$-analytic space $X$. 
If $X$ is reduced, then  $\metr$ is called {\it semipositive in $x \in X$} if $x$ has a  neighbourhood $V$ in $X$ with the following properties:
\begin{itemize}
\item[(i)] $V$ is a compact strictly $K$-analytic domain;
\item[(ii)]  $(V,\Lan|_V)$ has a formal $\kcirc$-model $(\Vcal, \Lcal)$ with $\metr_\Lcal = \metr$;
\item[(iii)] If $Y$ is a closed curve in $\Vcal_s$ with $Y$ proper over $\ktilde$, then $\deg_{\Lcal}(Y) \geq 0$.
\end{itemize}
If $X$ is not necessarily reduced, then $\metr$ is called {\it semipositive in $x$} if the induced metric on $L|_{X_{\rm red}}$ is semipositive in the above sense. 
We call $\metr$ {\it semipositive} on an open subset $W$ of $X$ if $\metr$ is semipositive in all points of $W$.  
\end{definition}

\begin{lem} \label{semipositivity and local model}
Suppose that $X$ is reduced and that $\metr$ is semipositive in $x \in X$. Let $W \subseteq V$  be compact strictly $K$-analytic domains in $X$ such that $W$ is a neighbourhood of $x$.  Suppose that $(\Vcal, \Lcal)$ (resp. $(\Wcal,\Mcal)$) is a formal $\kcirc$-model of $(V,\Lan|_V)$ (resp. $(W,\Lan|_W)$). If $(\Vcal,\Lcal)$ satisfies (ii) and (iii) in Definition \ref{semipositive formal metric in x} and if $(\Wcal,\Mcal)$ satisfies (ii), then $(\Wcal,\Mcal)$ also satisfies  (iii). 
 \end{lem}

\begin{proof}
We first note that we may always replace the formal $\kcirc$-models $\Vcal$ (resp. $\Wcal$) by dominating formal $\kcirc$-models $\Vcal'$ of $V$ (resp. $\Wcal'$ of $W$) as property (iii) is equivalent under such a change. This follows from the fact that any curve of $\Vcal_s$  is dominated by a curve $\Vcal_s'$ with respect to the proper morphism $\Vcal_s' \to \Vcal_s$ \cite[Cor. 4.4]{temkin-2000} and from projection formula. 
 In this way, Raynaud's theorem and \cite[Cor. 5.4]{bosch-luetkebohmert-2} show that we may assume that the $\rm G$-open immersion $W \to V$  extends to an open immersion $\Wcal \to \Vcal$. By Lemma \ref{formal metric and model}, we may assume that $\Vcal_s$ and hence $\Wcal_s$ are reduced, and then  $\Lcal|_\Wcal \cong \Mcal$ again by Lemma \ref{formal metric and model}. Therefore property (iii) for $(\Vcal,\Lcal)$ implies the same property for $(\Wcal,\Mcal)$.
\end{proof}

\begin{prop} \label{properties of semipositive metrics} 
Let $\metr$ be a piecewise linear metric on the line bundle $L$ on the strictly $K$-analytic space $X$ and let $x \in X$. 
\begin{itemize}
 \item[(a)] The set of points in $X$ where $\metr$ is semipositive is open.
 \item[(b)] The trivial metric on $O_X$ is semipositive on $X$.
 \item[(c)] The tensor product of two piecewise linear metrics which are semipositive in $x$ is again semipositive in $x$.
 \item[(d)] Let $f:X' \to X$ be a morphism of strictly $K$-analytic spaces  and let $x' \in X'$ with $x=f(x')$. If $\metr$ is semipositive in $x$, then $\metr':=f^*\metr$ is semipositive in $x'$.
\end{itemize}
\end{prop}

\begin{proof} 
We may assume that $X,X'$ are reduced. 
Properties (a) and (b) are obvious from Definition \ref{semipositive formal metric in x}. Lemma \ref{semipositivity and local model} and linearity of the degree of a proper curve with respect to the divisor shows (c). 

For (d), we choose $V$, $\Vcal$ and $\Lcal$ as in Definition \ref{semipositive formal metric in x}. Then 
there is a  compact $\rm G$-open neighbourhood $W$ of $x'$ in $(X')^{\rm an}$ with $f(W) \subseteq V$. By Raynaud's theorem, the  morphism $f:W \to V$ extends to a morphism $\varphi:\Wcal \to \Vcal$ of formal $\kcirc$-models. Let $Y$ be a closed curve in $\Wcal_s$ which is proper over $\ktilde$. Then the restriction of $\varphi$ to $Y$ is proper.  By \eqref{line bundle and pull-back metric}, we have $\metr'=\metr_{\varphi^*\Lcal}$, and (d) follows from projection formula. 
\end{proof}

In the following,  $L$ is a line bundle on  {a} proper algebraic variety $X$ over $K$.

\begin{prop} \label{local to global for semipositive}
We assume that $\metr$ is the formal metric associated to  {a} formal $\kcirc$-model $(\Xcal,\Lcal)$ of $(X,L)$ and we denote by $\pi:\Xan \to \Xcal_s$ the reduction map. Then $\metr$ is semipositive on the open subset $W$ of $\Xan$ if and only if $\deg_{\Lcal}(Y) \geq 0$ for any closed curve $Y$ in $\Xcal_s$ with $Y \subseteq \pi(W)$. 
\end{prop}

\begin{proof}
We note first that on the right hand side of the equivalence we may always pass to a formal $\kcirc$-model $\Xcal'$ dominating $\Xcal$ using that any curve in $\Xcal_s$  is dominated by a curve in $\Xcal_s'$ with respect to the proper morphism $\Xcal_s' \to \Xcal_s$  and using the projection formula.

Suppose that $\metr$ is semipositive on $W$ and let $Y$ be a closed curve in $\Xcal_s$ with $Y \subseteq \pi(W)$. There is $x \in W$ with $\pi(x)$ equal to the generic point of $Y$. Since $\metr$ is semipositive in $x$, there is a neighbourhood $V$ of $x$ in $\Xan$ and a formal $\kcirc$-model $(\Vcal,\Lcal)$ of $(V,\Lan|_V)$ with properties (i)--(iii) from Definition \ref{semipositive formal metric in x}.  Using Raynaud's theorem and \cite[Cor. 5.4]{bosch-luetkebohmert-2}, we may assume that $\Vcal$ is an open subset of $\Xcal$. 
Since 
 {$V$ is a neighborhood of $x$ in $\Xan$} and since $\Xan$ is boundaryless \cite[Thm. 3.4.1]{berkovich-book}, we deduce that $x$ is not a boundary point of $V$ and hence the closure of $\pi(x)$ in $\Vcal_s$ is proper (see \cite[Lemme 6.5.1]{chambert-loir-ducros}). We conclude that this closure is $Y$ and hence $\deg_\Lcal(Y) \geq 0$ by (iii). 

Conversely, assume that $\deg_{\Lcal}(Y) \geq 0$ for any closed curve $Y$ in $\Xcal_s$ with $Y \subseteq \pi(W)$. For $x \in W$, we choose a neighbourhood $V$ of $x$ in $W$ such that $V$ is a compact strictly $K$-analytic domain. By  Raynaud's theorem and \cite[Cor. 5.4]{bosch-luetkebohmert-2}, we may assume that $V$ has a $\kcirc$-model $\Vcal$ which is a formal  open subset of $\Xcal$. Then (iii) in Definition \ref{semipositive formal metric in x} follows from our assumption on the degree of curves since $\Vcal_s \subseteq \pi(W)$. This proves semipositivity of $\metr$ in $x$. 
\end{proof}

\begin{rem} \label{semipositive and rational point}
It follows that any formal metric $\metr$ is semipositive in all $K$-rational points of $X$. Indeed, using the notation from Proposition \ref{local to global for semipositive}, we note that that the reduction $\pi(x)$ is a closed point of the special fibre $\Xcal_s$. By anticontinuity of the reduction map $\pi$, we get an open neighbourhood $W:=\pi^{-1}(\pi(x))$ of $x$ in $\Xan$ for which no closed curve of $\Xcal_s$ is contained in $\pi(W)$. Then Proposition \ref{local to global for semipositive} proves semipositivity of $\metr$ on $W$.
 \end{rem}

\begin{art} \label{CL measure}
Let $(L,\metr)$ be a formally metrized line bundle on  {an} $n$-dimensional proper variety $X$ over $K$ and let $W$ be an open subset of $\Xan$. 
Then the $\delta$-form $c_1(L,\metr)^n$ of type $(n,n)$ induces a unique Radon measure on $W$ extending the current $[c_1(L,\metr)^n]_D$ (see \cite[Cor. 6.15]{gubler-kuennemann}). 
It is shown in \cite[Thm. 10.5]{gubler-kuennemann} that this Monge--Amp\`ere measure on $\Xan$ agrees with the corresponding Chambert--Loir measure in arithmetic geometry and hence it is supported in finitely many points. Note that the restriction of this measure to $W$ is the unique Radon measure  on $W$ extending the current $[c_1(L|_W,\metr)^n]_D$.  
\end{art}

\begin{lem} \label{curveslemma}  \label{formal and psh on a curvelemma}
Let $L$ be a line bundle on a proper curve $C$ over $K$ endowed with a formal metric $\metr$ and let $W$ be an open subset of ${C^{\rm an}}$. 
Then 
 $\metr$ is a semipositive formal metric over $W$ if and only if $c_1(L|_W,\metr)$ induces a positive measure on $W$.
\end{lem}

\begin{proof}
Let $\KL$ be a formal model of $L$ over a proper formal model $\Xcal$
of $C$ which induces the given metric and let $\pi:{C^{\rm an}} \to \Xcal_s$ be the reduction map.
By  \ref{reduced special fibre}, we may assume $\Xcal_s$ reduced. 
The equality of the Monge--Amp\`ere measure and the Chambert-Loir measure \cite[Thm. 10.5]{gubler-kuennemann}
gives the formula
\begin{equation} \label{CL measure for curves}
c_1(L,\metr)=\sum_Y\deg (\KL|_Y)\cdot \delta_{\xi_Y}
\end{equation}
where $Y$ runs over the irreducible components of the special fibre
of $\Xcal$ and $\xi_Y$ denotes the corresponding point in the
Berkovich space $C^\an$ with reduction equal to the generic point of $Y$.
In view of \eqref{CL measure for curves} and using Proposition \ref{local to global for semipositive}, the lemma follows from the claim that $\xi_Y \in W$ if and only if $Y \subseteq \pi(W)$. 

To see the equivalence, let $\xi_Y \in W$. Then $\xi_Y$ has a compact  strictly $K$-analytic domain $V \subseteq W$ as a neighbourhood with a formal $\kcirc$-model $\Vcal$ of $V$. By Raynaud's theorem, we may assume that the inclusion $V \to {C^{\rm an}}$ extends to a morphism $\iota:\Vcal \to \Xcal$. 
Since $\xi_Y$ is an inner point of $V$, the closure $Y'$ of the reduction of $\xi_Y$ in $\Vcal$ is proper over $\ktilde$ \cite[Lemme 6.5.1]{chambert-loir-ducros} and hence $\iota(Y')$ is a proper curve over $\ktilde$. By functoriality of the reduction map, we have $\pi(\xi_Y) \in \iota(Y') \subseteq \pi(V) \subseteq \pi(W)$.  Since $\pi(\xi_Y)$ is dense in $Y$ and  $\iota(Y')$ is proper, we get $ Y=\iota(Y') \subseteq \pi(W)$. 

Conversely, if $Y \subseteq W$, then there is $x \in W$ with $\pi(x)$ equal to the generic point of $Y$. By the characterization of $\xi_Y$, we get $\xi_Y=x \in W$.
\end{proof}

We recall that a line bundle is called semiample if a 
strictly positive tensor-power is generated by global sections. 
Note also that a formal metric $\metr$ on $L$ has a canonical 
first Chern $\delta$-form $c_1(L,\metr) \in B^{1,1}(\Xan)$ 
\cite[Rem. 9.16]{gubler-kuennemann}. 

\begin{lem} \label{local ample Modellmetrik ist positiv repraesentierbar}
Let $V$ be a compact strictly $K$-analytic neighbourhood of $x$ in $\Xan$. Suppose that the formal metric $\metr$ on $L|_V$ is induced by a formal model $(\Vcal, {\Lcal_V})$ of $(V,\Lan|_V)$ and let $Y$ be the closure of the reduction of $x$ in $\Vcal_s$.  If the restriction of $ {\Lcal_V}$ to $Y$ is semiample,  then the first Chern $\delta$-form $c_1(L,\metr)$ is positively representable in the sense of \ref{positive delta-forms on W} on an open neighbourhood of $x$.
\end{lem}

\begin{proof}
We will show that $x$ is contained in a tropical chart $(V,\varphi_U)$ 
with $V \subseteq W$ such that $c_1(L|_V,\metr)$ is induced
by a positively representable element $\alpha$ in 
$P^{1,1}(V,\varphi_U)$.
Recall from \ref{positivity on tropical charts}(iii)
that this means that $\alpha$ admits a positive representative 
$\tilde\alpha$ which is a positive $\delta$-preform on $N_{U,\R}$. 
We may always replace $\Vcal$ by a dominating formal $\kcirc$-model 
of $V$. 
 {By Proposition \ref{extension of PL-metrics},}  
we may assume that $\Vcal$ is a formal open subset of a formal $\kcirc$-model $\Xcal$ of $\Xan$ such that  {$\Lcal_V$ extends} to a line bundle on $\Xcal$.
The formal GAGA theorem of Fujiwara--Kato \cite[Thm. I.10.1.2]{fujiwara-kato-1} shows that $\Xcal$ is dominated by a (proper flat) algebraic $\kcirc$-model of $X$ (see also the proof of \cite[Prop. 8.13]{gubler-kuennemann}) and that the pull-back of $\Lcal$ is an algebraic line bundle. So we may assume that $\Xcal$ and $\Lcal$ are both algebraic $\kcirc$-models.

Replacing $\Lcal$ by a  positive tensor power, there is  a generating set $\{\tilde s_1, \dots, \tilde s_n\}$ of global sections of $\Lcal|_{Y}$. 
Since $x$ is an inner point of $V$, the closure $Y$ of the reduction $\pi(x)$ of  $x$ in $\Vcal_s$ is proper \cite[Lemme 6.5.1]{chambert-loir-ducros} and hence  $Y$ is also the closure of $\pi(x)$ in $\Xcal_s$. By anticontinuity of the reduction map $\pi: \Xan \to \Xcal_s$, the subset $\pi^{-1}(Y)$ of $V$ is an open neighbourhood of $x$ in  $\Xan$.
 
 We cover $Y$ by finitely many trivializations $(\Ucal_i)_{i=1,\dots,t}$ of $\Lcal$ with special fibre $(\Ucal_i)_s$ contained  in $\Vcal_s$ and intersecting $Y$. We may assume that  there are meromorphic (algebraic) sections $s_i$ of $\Lcal$ which restrict to invertible  sections on $\Ucal_i$ and agree  with $\tilde s_i$ on $Y$. In particular, there is an open subset $U$ of $X$ such that $x \in \Uan$ and such that every $s_i$ restricts to an invertible section of $L|_U$. Since tropical charts form a basis for the topology on $\Xan$, we may assume that $U$ is very affine and that $(W,\varphi_U)$ is a tropical chart with $x \in W \subseteq \pi^{-1}(Y) \subseteq V$. 
For any $w \in W$, we have $\|s_i(w)\| \leq 1$ for all  $i \in \{1,\dots,t\}$ using that $s_i$ restricts to a global section on $Y$. Moreover, there is an $i$ such that $\pi(w) \in (\Ucal_i)_s$ and hence $\|s_i(w)\|=1$. 
For  a fixed frame $s$ of $L|_U$, we get 
$$
\|s(w)\| = \max_i \left|\frac{s}{s_i}(w)\right|.
$$
We consider the character lattice $M_U=\Ocal(U)^*/K^*$ of the torus $T$ associated to $U$. 
By definition, we have $u_i:=\frac{s}{s_i} \in M_U = N_U^*$ 
and hence $c_1(L|_V,\metr)$ is represented by the 
$\delta$-preform
$\tilde\alpha := d'd'' [\max_i u_i]$ on $N_{U,\R}$. 
Since $\max_i u_i$ is a convex function, 
Example \ref{convex function} yields that  $\tilde\alpha$ 
is a positive $\delta$-preform on $N_{U,\R}$.
\end{proof}

\begin{thm}\label{main theorem for pl}
Let $L$ be a  line bundle on an algebraic variety $X$ over $K$ and let  $W$ be an open subset of $\Xan$. Then the following properties are equivalent for a piecewise linear metric $\metr$ on $L$ over $W$:
\begin{itemize}
 \item[(1)] The piecewise linear metric $\metr$ is semipositive on $W$.
 \item[(2)] The metric is functorial $\delta$-psh.
 \item[(3)] The metric $\metr$ is functorial psh.
 \item[(4)] The $\delta$-form $c_1(L|_W,\metr)$ is positive on $W$.
 \item[(5)] The restriction of $\metr$ to $W \cap C^{\rm an}$  is psh for any closed curve $C$ of $X$.
\end{itemize}
\end{thm}

There is also an equivalent version of Theorem \ref{main theorem for pl} in terms of formal metrics. This version was given in Theorem \ref{localnotes5}. To see that their equivalence, we note that we may assume $X$ proper over $K$ by Nagata's compactification theorem \cite{nagata-1962}. Since (1)--(5) are local in the analytic topology, we may assume that $\metr$ extends to a formal metric on $L$. This shows the desired equivalence. 

\begin{proof}[ {Proof of Theorem \ref{localnotes5}}]
The above remark shows that we may assume that $\metr$ extends to a metric on $L$ which we also denote by $\metr$. 
Let $(\Xcal,\Lcal)$ be a formal $\kcirc$-model of $(X,L)$ with $\metr = \metr_\Lcal$ over $\Xan$.

(1) $\Rightarrow$ (2): 
Since semipositivity of formal metrics is functorial, it is enough to show that $[c_1(L|_W,\metr)]_E$ is a positive $\delta$-current. By the projection formula \cite[Prop. 5.9(iii)]{gubler-kuennemann}, we may  check that on a generically finite projective covering and so we may assume that $X$ is projective using Chow's lemma. 

For any $x \in W$, there is a compact strictly $K$-analytic domain $V$ as a neighbourhood in $W$ such that $(V,\Lan|_V)$ has a formal $\kcirc$-model $(\Vcal, \Mcal= \Lcal|_\Vcal)$.  Since $\metr$ is semipositive in $x$, we may choose this model such that $\deg_\Mcal(Y) \geq 0$ for all curves $Y \subseteq \Vcal_s$ which are proper over $\ktilde$. 
Let $Z$ be the closure of the reduction of $x$ in $\Vcal_s$. Since $x$ is an inner point of $V$, the variety $Z$ is proper over $\ktilde$ \cite[Lemme 6.5.1]{chambert-loir-ducros}. By construction, the restriction of $\Mcal$ to $Z$ is nef.

By Lemma \ref{semipositivity and local model}, we may always pass to a dominating formal $\kcirc$-model of $\Vcal$. Using Raynaud's theorem and \cite[Cor. 5.4]{bosch-luetkebohmert-2}, we may assume that $\Vcal$ is a formal open subset of $\Xcal$. 
By \cite[Prop. 10.5]{gubler-pisa}, $\Xcal$ is dominated by the formal completion of a projective flat $\kcirc$-model  and hence we may assume that $\Xcal$ is projective. Then the formal GAGA-theorem of Ullrich \cite[Thm. 6.8]{ullrich} shows that $\Lcal$ is an algebraic line bundle as well.

We fix a very ample line bundle $\Hcal$ on $\Xcal$ with generic fibre $H$. Let $\metr_H$ be the semipositive algebraic metric on $H$ given by the very ample model $\Hcal$. 
For every rational $\ve>0$, the $\Q$-line bundle $L_\ve :=L \otimes H^\ve$ has the  metric $\metr_\ve:=\metr \otimes \metr_H^\ve$ over $W$ given by the model $\Lcal_\ve:=\Lcal \otimes \Hcal^\ve$ on $\Xcal$.

The restriction of $\Lcal_\ve$ to $Z$ is ample  as it is the tensor product of an ample line bundle with a nef line bundle. 
By Lemma \ref{local ample Modellmetrik ist positiv repraesentierbar}, the  $\delta$-form $c_1(L,\metr_\ve)$ is positively representable on an open  neighbourhood of $x$ in $W$. As $x$ was any point of $W$, we conclude that $c_1(L|_W,\metr)$ is positively representable on $W$. By Corollary \ref{delta forms vs delta currents}, the associated $\delta$-current $[c_1(L|_W,\metr)]_E$ is positive and hence  
$$\langle [c_1(L|_W, \metr)]_E,\beta\rangle +  \ve \langle [c_1(H|_W,\metr_H)]_E, \beta \rangle 
= \langle [c_1(L_\ve|_W, \metr_\ve)]_E,\beta\rangle \geq 0$$
for any positive $\beta\in B_c^{n-1,n-1}(W)$ where $n:=\dim(X)$. Using $\ve  \to 0$, we deduce $\langle [c_1(L|_W, \metr)]_E,\beta\rangle \geq 0$ and hence  $ [c_1(L|_W, \metr)]_E $ is positive proving (2).

(2) $\Rightarrow$ (3):  Since any positive $\delta$-current is a positive current (see Remark \ref{positive currents on W}), this is obvious.

(3) $\Leftrightarrow$ (4): This follows from Proposition \ref{functorial current positive and positive delta-forms}.

(3) $\Rightarrow$ (5): This is obvious.

(5) $\Rightarrow$ (1): By Remark \ref{properties of algebraic metrics}, we may assume that $\metr$ is given by an algebraic $\kcirc$-model $(\Xcal,\Lcal)$ of $(X,L)$. Let $Y$ be a closed curve in $\Xcal_s$ with  $Y \subseteq \pi(W)$. Since $\Xcal$ is proper, it is clear that $Y$ is proper over the residue field $\ktilde$. By Theorem \ref{liftvarieties}, there is a closed curve $C$ in $X$ whose closure in $\Xcal$ has $Y$ as an irreducible component. 
We look again at the discrete Radon measure which extends the current $[c_1(L|_{C^{\rm an} },\metr)]_D$ on $C^{\rm an}$. Positivity of $[c_1(L|_{W \cap C^{\rm an}},\metr)]_D$  yields positivity of the discrete Radon measure. 
By Lemma \ref{curveslemma}, we deduce that the restriction of $\metr$ to $C^{\rm an} \cap W$ is semipositive.  This means in particular $\deg_\Lcal(Y) \geq 0$ proving that $\Lcal$ is vertically nef and (1).
\end{proof}

\section{Semipositive approximable metrics} \label{Section: Semipositive approximable metrics} 

We consider a metric $\metr$ 
on a line bundle $L$ of a proper variety $X$ over $K$.

\begin{art} \label{semipositive approximable metrics}
We say that  $\metr$  is 
\emph{semipositive approximable} if it is the uniform limit of a
sequence $(\metr_n)_{n\in \N}$ of  semipositive $\Q$-formal metrics on $L^\an$. 

This class of metrics was introduced by Zhang and includes canonical metrics of dynamical systems.
It is clear that every semipositive formal metric is semipositive approximable. 
 {In this section, we will prove the converse. This was proven in the special case of a  discretely valued field with residue characteristic $0$ by Boucksom, Favre and Jonsson  \cite[Remark after Thm. 5.12]{boucksometal-2012}.
Our proof works in general and is based on the {Theorem} \ref{liftvarieties}. 
Amaury Thuillier told us that he has a similar proof.}
\end{art}

\begin{prop} \label{semipositive approximable formal metrics}
Suppose that $\metr$ is a formal metric. Then $\metr$ is semipositive approximable if and only if it is semipositive.
\end{prop}

\begin{proof} We have to show that a semipositive approximable formal metric $\metr$ is semipositive. Let $\metr_n$ be semipositive $\Q$-formal metrics on $L^\an$  approximating the formal metric $\metr$ uniformly. By Remark \ref{properties of algebraic metrics}, there is an algebraic $\kcirc$-model  $(\Xcal,\Lcal)$  of $(X,L)$ with $\metr = \metr_\Lcal$. Let $V$ be  a closed curve contained in $\Xcal_s$. Then Theorem \ref{liftvarieties} shows that there is a closed curve $Y$ in $X$ such that $V$ is an irreducible component of the special fibre of the closure $\overline{Y}$ in $\Xcal$. The restriction of the metrics $\metr_n$ to $Y$ are semipositive  and the Chambert--Loir measures $c_1(L|_Y,\metr_n)$ converge weakly to $c_1(L|_Y,\metr)$. We conclude that $c_1(L|_Y,\metr)$ is a positive discrete measure.  By Lemma \ref{curveslemma}, the restriction of $\metr$ to $L|_Y$ is semipositive and hence $\deg_\Lcal(V)\geq 0$. This proves semipositivity of the formal metric $\metr$. 
\end{proof}

\begin{prop} \label{semipos-approx yields functorial psh}
{If $\metr$ is semipositive approximable, then the pull-back of $\metr$ to any curve is psh.}
\end{prop}

\proof Semipositive approximable is stable under pull-back, so it is enough to show that $\metr$ is psh in the case of a curve $X$. By Corollary \ref{descent of psh}, it is enough to show that the pull-back metric is psh on $X'$ for a proper 
surjective morphism $X' \to X$ of curves. We conclude that we may assume $X$ projective over $K$. 
Then the first Chern current $[c_1(L,\metr)]_D$ is induced by the corresponding Chambert--Loir measure $c_1(L,\metr)$. By definition, the latter is the weak limit of positive discrete measures and hence $c_1(L,\metr)$ is also a positive measure. This means that the first Chern current $[c_1(L,\metr)]_D$ is positive proving the claim. 
 \qed

\section{Piecewise smooth metrics} \label{Piecewise smooth metrics}

We have introduced piecewise smooth metrics on line bundles in \cite[\S 8]{gubler-kuennemann}. They include smooth metrics, formal metrics and 
  canonical metrics. In this section, we relate them to the positivity notions from Section \ref{pluri}.

\begin{art} \label{piecewise smooth and corner locus}
Let $C = (\Ccal,m)$ be a tropical cycle on $N_\R$ for a lattice $N$ of finite rank and let $\Omega$ be an open subset of the support $|\Ccal|$.  We say that $\phi: \Omega \to \R$ is a {\it piecewise smooth} function if there is an integral $\R$-affine polyhedral subdivision $\Dcal$ of $\Ccal$ and  smooth functions $\phi_\sigma: \Omega \cap \sigma \to \R$ with  $\phi|_{\Omega \cap \sigma} = \phi_\sigma$ for every $\sigma \in \Dcal$. 
 
In a similar way as above, piecewise smooth superforms on $\Omega$ are defined in \cite[3.10]{gubler-kuennemann}. In particular, we get a piecewise smooth superform $\dpa \phi$ (resp. $\dpb \phi$) given by the superform $d' \phi_\sigma$ (resp. $d'' \phi_\sigma$) on $\Omega \cap \sigma$ for every $\sigma \in \Ccal$.

We recall from \cite[1.10--1.12]{gubler-kuennemann} that  a piecewise smooth function $\phi$ on $|\Ccal|$ induces a tropical cycle $\phi \cdot C$ of codimension $1$ in $|\Ccal|$ called the {\it corner locus of $\phi$}. Its support is the non-differentiability locus of $\phi$ and its smooth weights are defined in terms of the outgoing slopes of $\phi$. The above notions are related by  
\begin{equation} \label{tropical PL}
d'd''[\phi] = [\dpa \dpb \phi] + \delta_{\phi\cdot \Trop (U)} \in D^{1,1}(|\Ccal|) 
\end{equation}
as a consequence of the tropical Poincar\'e--Lelong formula (see \cite[Cor. 3.19]{gubler-kuennemann}), where $[\alpha]$ denotes
 the supercurrent associated to a piecewise smooth $\alpha$.
\end{art}

\begin{art} \label{tropical frame}
{Let $L$ be a line bundle on  {an} algebraic variety $X$ over $K$. We recall that a metric $\metr$ on $L$ over an open subset $W$ of $\Xan$ is called {\it piecewise smooth} if 
for any $x \in W$ there is a tropical chart $(V,\varphi_U)$ with $x \in V \subseteq W$, a frame  $s$ of $L$ over $U$ and
a piecewise smooth function $\phi:\Omega\to\R$ 
such that
$-\log \|s\||_V=\phi\circ \trop_{U}|_V$. 
Here, the open subsets 
$\Omega:=\trop_{U}(V)$ of $\Trop(U)$ 
 {may be assumed to 
have convex intersection with all faces of
$\Trop(U)$} and 
we may assume that $\phi$ extends to a piecewise smooth function 
$\tilde\phi:N_\R \to \R$. 
In this case, we will call $(V,\varphi_U,\Omega,s,\phi)$ a {\it 
tropical frame} for the piecewise smooth metric $\metr$.} 

The choice of $\tilde\phi$ will be not be important for the following. We need this extension only to make the corner locus $\tilde\phi \cdot \Trop(U)$ well-defined as a tropical cycle contained in $\Trop(U)$. The definition of the weights of the corner locus in \cite[1.10]{gubler-kuennemann} shows  that the restrictions of the weights to 
$\Omega$ depend only on $\phi$. 
We have therefore decided to drop $\tilde\phi$ from our 
notation for tropical frames.
\end{art}

\begin{art} \label{canonical decomposition}
Let $\metr$ be a piecewise smooth metric on $L$ over $W$. We recall from \cite[9.5--9.8]{gubler-kuennemann} that there is a canonical piecewise smooth form $c_1(L,\metr)_{\rm ps}$ on $W$ and a canonical generalized $\delta$-form $c_1(L,\metr)_{\rm res} \in P^{1,1}(W)$ of codimension $1$ such that for the associated $\delta$-currents, we have 
\begin{equation} \label{candeco} 
[c_1(L|_W,\metr)]_E = [c_1(L|_W,\metr)_{\rm ps}]_E + [c_1(L|_W,\metr)_{\rm res}]_E
\end{equation}
in a functorial way. If $(V,\varphi_U,\Omega,s,\phi)$  is a tropical frame for $\metr$, then $c_1(L,\metr)_{\rm ps}$ is given on the tropical chart $(V,\varphi_U)$ by $\dpa \dpb \phi$  
and the generalized $\delta$-form $c_1(L,\metr)_{\rm res}$ is represented on $(V,\varphi_U)$ by the $\delta$-preform $d'd''\delta_{\tilde \phi \cdot N_{U,\R}} \in P^{1,1}(N_{U,\R})$. 
\end{art}

\begin{thm}\label{psh slope condition}
Let $L$ be a line bundle on  {an} algebraic variety $X$ over $K$.
Let $\metr$ be a piecewise smooth metric on  $L$ over an open subset $W$ of
$X^\an$. 
Then the metric $\metr$ is plurisubharmonic if and only if 
{for each tropical frame $(V,\varphi_U,\Omega,s,\phi)$ of $\metr$} we have
\begin{enumerate}
\item[(i)]
the restriction of $\phi$ to each maximal face of $\Trop (U) {\cap \Omega}$ is a convex function
and
\item[(ii)]
the corner locus $\tilde\phi\cdot \Trop (U)$ is effective on $\Omega$ .
\end{enumerate}
\end{thm}

\proof 
Let $n:=\dim(X)$ and let $(V,\varphi_U,\Omega,s,\phi)$ be a tropical frame for $\metr$.  
A positive superform $\alpha_U \in A^{n-1,n-1}(\Omega)$  
induces a positive $(n-1,n-1)$-form $\alpha$ on $V$. 
Note that $\alpha$ has compact support in $V$ if and only 
if $\alpha_U$ has compact support in $V$ 
(see \cite[Cor. 3.2.3]{chambert-loir-ducros}).  
Assuming $\alpha$ with compact support and using
\begin{equation} \label{ps0}
\langle [c_1(L|_W,\metr)]_D , \alpha \rangle =  \langle d'd''[-\log \|s\|]_D, \alpha \rangle =  \langle d'd''[\phi], \alpha_U \rangle,
\end{equation}
the tropical Poincar\'e--Lelong formula \eqref{tropical PL} yields 
\begin{equation} \label{ps1}
\langle [c_1(L|_W,\metr)]_D , \alpha \rangle = \langle [\dpa \dpb \phi], \alpha_U \rangle + \langle \delta_{\tilde\phi\cdot \Trop (U)}, \alpha_U \rangle. 
\end{equation}
We first assume that $\metr$ is plurisubharmonic which means that the first Chern current $[c_1(L|_W,\metr)]_D$ is positive. Let $\alpha_U$ be any positive superform of bidegree $(n-1,n-1)$ with compact support in $\Omega$ and let $\alpha$ be the induced smooth form on $V$. Since $\alpha$ is positive, we deduce that \eqref{ps0} is non-negative and hence  {$d'd''[\phi ]$} is a positive supercurrent on $\Omega$. We deduce that the piecewise smooth extension $\tilde\phi$ induces a positive supercurrent on $\Omega \cap \relint(\Delta)$ for any maximal face $\Delta$ of $\Trop(U)$. By Example \ref{convex function},   $\phi$ is a convex function on $\Omega \cap \Delta$ proving (i).  
 {Note that the support of the corner locus $\tilde\phi\cdot \Trop (U)$ is of smaller dimension than the maximal faces of $\Trop(U)$ and hence we may alter $\alpha_U$ outside this support to a positive superform $\alpha_U'$ such that $\langle [\dpa \dpb \phi], \alpha_U' \rangle$ is arbitrarily small. Hence we deduce from  
\eqref{ps1} that $\delta_{\tilde\phi\cdot \Trop (U)}$ is a positive supercurrent on $\Omega$. 
Applying 
Example \ref{polyhedral supercurrent} to the maximal faces of $\Trop(U)$, we get (ii).}

Conversely, we assume that (i) and (ii) are always satisfied. 
Given a positive smooth form $\alpha$ on $\Xan$ with compact support in $W$ and of bidegree $(n-1,n-1)$,
we have to show that  $\langle [c_1(L|_W,\metr)]_D ,\alpha\rangle\geq 0$. 
We cover the support of $\alpha$ by finitely many non-empty tropical frames 
$(V_i,\varphi_{U_i},\Omega_i,s_i,\phi_i)$, $i=1, \dots,  {l}$, such 
that $\alpha$ is given on $V_i$ by $\alpha_i \in A^{n-1,n-1}(\Omega_i)$ 
for $\Omega_i:=\trop_{U_i}(V_i)$. Then we consider a non-empty very affine 
open subset $U \subseteq U_1 \cap \cdots \cap  {U_l}$. 
Then there is a unique $\alpha_U \in A^{n-1,n-1}(\Trop(U))$ such 
that $\alpha$ is given on $\Uan$ by $\alpha_U$ (as the argument in 
\cite[Prop. 5.13]{gubler-forms} shows, see also 
\cite[Prop. 5.7]{gubler-kuennemann}). 
Note that the support of $\alpha_U$ is not necessarily compact, 
but it is contained in 
$\Omega:=  {\bigcup_{i=1}^l} F_i^{-1}( {\Omega_i})$ for the 
canonical integral $\R$-affine maps $F_i:N_{U,\R} \to N_{U_i,\R}$. 
We define $\phi: \Omega \to \R$ on $F_i^{-1}( {\Omega_i})$ by 
$\phi:= \phi_i \circ F_i$ and $V:=\trop_U^{-1}(\Omega)$. 
Then $V$ contains $\supp(\alpha) \cap \Uan$. 
Note that $c_1(L|_W,\metr)_{\rm res} \wedge \alpha$ and 
$c_1(L|_W,\metr)_{\rm ps} \wedge \alpha$ are of type $(n,n)$ 
and hence their support is contained in $\Uan$ 
(see \cite[Cor. 5.6, 9.5]{gubler-kuennemann}). 
We conclude that both supports are contained in a compact subset $C$ of $V$. 
By 
{continuity} of the tropicalization map, we get that $\trop_U(C)$ 
is a compact subset of $\Omega$. We may shrink $\Omega$ a bit still 
containing $\trop_U(C)$ such that $\phi$ is the restriction of a 
piecewise smooth function $\tilde\phi$ on $N_{U,\R}$.  
Then $(V,\varphi_U, \Omega, s, \phi)$ is a tropical frame for $\metr$. 
Evaluating \eqref{candeco} at $\alpha$ gives 
\begin{equation} \label{ps2}
\langle [c_1(L|_W,\metr)]_D , \alpha \rangle 
= \int_{ {|\Trop(U)|}} \dpa \dpb \phi \wedge \alpha_U  
+ \int_{|\Trop(U)|}\delta_{\tilde\phi\cdot \Trop (U)} \wedge \alpha_U. 
\end{equation}
Since $\alpha$ is positive, the superform $\alpha_U$ is positive (see Proposition \ref{local positivity is fine} and Remark \ref{compare cld}). 
 {By (i), $\phi$ is convex and hence Example \ref{convex function} yields that}  
$\dpa \dpb \phi$ restricts to a positive superform on any maximal face of $\Trop(U) \cap \Omega$ where $\phi$ is smooth. 
 {Then  Proposition \ref{positivity properties of superforms} shows} 
that $ \dpa \dpb \phi \wedge \alpha_U$ is also positive on any such maximal face of $\Trop(U) \cap \Omega$. It follows that the first integral in \eqref{ps2} is non-negative. 
 {By (ii), the smooth weight function of a maximal face $\sigma$ of the corner locus  $\tilde\phi\cdot \Trop (U)$ is non-negative on $\sigma \cap \Omega$ and the positivity of the superform $\alpha_U$ on $\Omega$ yields that $\alpha_U|_{\Omega \cap \sigma}$ is a positive $(n-1,n-1)$-form. We conclude that the second integral is non-negative  and hence \eqref{ps2} is non-negative.} 
\qed

\begin{rem}\label{rem psh slope condition}
Let $L$ be a line bundle on  {an} algebraic variety $X$.
Let $\metr$ be a piecewise smooth metric on  $L$ over an open 
subset $W$ of $X^\an$. 

(a) Our proof of Theorem \ref{psh slope condition} shows as well
that $\metr$ is already plurisubharmonic if $W$ admits a basis of
tropical frames $(V,\varphi_U,\Omega,s,\phi)$ for $\metr$ such that
conditions (i) and (ii) in Theorem \ref{psh slope condition} hold. 
We may then even replace (i) by the seemingly weaker condition:
\begin{itemize}
\item[(i')] \emph{there is a polyhedral complex $\Ccal$ with $|\Ccal|=|\Trop(U)|$ such that  $\phi|_{\Omega \cap \Delta}$ is a smooth convex function for any maximal face $\Delta$ of $\Ccal$.} 
\end{itemize}

(b) Let $(V,\varphi_U,\Omega,s,\phi)$ be a tropical frame of 
$\metr$. 
Let $f:X'\to X$ be a morphism of varieties and $(V',\varphi_{U'})$
a tropical chart on $X'$ such that 
 {the intersection of every face of $\Trop(U')$ with $\Omega'={\trop}_{U'}(V')$}
is convex, $f(U')\subseteq U$ and
$f^\an(V')\subseteq V$. Let $F:N_{U',\R}\to N_{U,\R}$ be the 
integral $\R$-affine map induced by $f$, let $\phi'$ be the restriction of $\phi \circ F$ to $\Omega'$ 
and let us consider the frame $s':=f^*(s)|_{U'}$ of $L':=f^*(L)$ over $U'$. Then $(V',\varphi_{U'},\Omega',s',\phi')$ is 
a tropical frame for the piecewise smooth metric $f^*\metr$ of $L$ over $f^{-1}(W)$. 

Assume that  $(V,\varphi_U,\Omega,s,\phi)$ satisfies condition
\ref{psh slope condition}(i). Then condition \ref{psh slope condition}(i)  holds for {$(V',\varphi_{U'},\Omega',s',\phi')$} as well. 

(c) Part (b) applies to the special case when $(V', \varphi_{U'})$ is also a tropical chart of $X$ with $V' \subseteq V$ and $U' \subseteq U$. 
Then $f = \id_X$, $L'=L$ and $s'=s$. 
We conclude that the tropical frame $(V',\varphi_{U'},\Omega',s,\varphi')$ satisfies \ref{psh slope condition}(i)
if the tropical frame $(V,\varphi_{U},\Omega,s,\varphi)$ 
satisfies \ref{psh slope condition}(i).
Observe that it is not clear whether an analog of the above 
statement holds for condition (ii) in Theorem \ref{psh slope condition}. To achieve that, we have to work below with positive representable $\delta$-preforms.

(d) Let $(V,\varphi_U,\Omega,s,\phi)$ be a tropical frame of 
$\metr$.
Instead of \ref{psh slope condition}(ii) we may consider the following stronger condition:
\begin{itemize}
\item[(ii')] 
\it{there exists a piecewise smooth 
function $\tilde\phi:N_{U,\R}\to \R$ with $\tilde\phi|_\Omega=\phi$ 
such that 
the  corner locus $\tilde\phi\cdot N_{U,\R}$ defines a positive
$\delta$-preform on $\tilde\Omega$ for some open $\tilde\Omega$
in $N_{U,\R}$ with $\Omega =\tilde\Omega\cap \Trop(U)$.}
\end{itemize}
In the setup of part (b)
we conclude from Proposition \ref{positivity properties of preforms}(c) that the tropical frame $(V',\varphi_{U'},\Omega',s',\varphi')$
of $f^*\metr$ fulfills conditions (ii')
if the tropical frame $(V,\varphi_{U},\Omega,s,\varphi)$ 
satisfies condition (ii').
\end{rem}

\begin{ex}\label{slope condition for curves}
Assume in Theorem \ref{psh slope condition} that $X$ is a curve. Then $\Trop(U)$ is a metrized graph, where the length 
of a primitive vector of an edge is defined as $1$, and $\Omega$ is an open subset of $\Trop(U)$. In this case, condition (i) and (ii) can be summarized by the condition that for any $\omega \in \Omega$, the sum of the outgoing slopes of $\phi$ at $\omega$ (along the finitely many edges emerging from $\omega$) is non-negative.
\end{ex}

The following result gives a sufficient condition for a 
piecewise smooth metric to be plurisubharmonic. 
It can be checked on a given covering by tropical charts.

\begin{prop}\label{second psh slope condition}
Let $L$ be a line bundle on  {an} algebraic variety $X$.
Let $\metr$ be a piecewise smooth metric on  $L$ over an open subset $W$ of
$X^\an$. 
Then the metric $\metr$ is functorial $\delta$-psh if $W$ admits a covering
by tropical frames $((V_i,\varphi_{U_i},\Omega_i,s_i,\phi_i))_{i\in I}$ for $\metr$ satisfying 
the  {following conditions for all $i \in I$:} 
\begin{enumerate}
\item[(i')]
{there is a polyhedral complex $\Ccal$ with $|\Ccal|=|\Trop(U_i)|$ such that $\phi_i|_{\Delta \cap \Omega}$ is smooth and convex for every maximal face $\Delta$ of $\Ccal$,} 
\item[(ii')]
$\phi_i$ is the restriction of a piecewise smooth function $\tilde\phi_i:N_{U,\R}\to \R$ as in Remark \ref{rem psh slope condition}(d) such that 
$\tilde\phi_i\cdot N_{U_i,\R}$ defines a positive
$\delta$-preform on $\tilde\Omega_i$.
\end{enumerate}
In particular, the metric $\metr$ is then plurisubharmonic.
\end{prop} 
\begin{proof}
We have seen in Remark \ref{rem psh slope condition} that conditions (i') and (ii') are functorial, so it is enough to show that the $\delta$-current $[c_1(L|_W,\metr)]_E$ is positive. Let $n:=\dim(X)$ and let $\alpha$ be a 
positive $\delta$-form of type $(n-1,n-1)$ with compact support in $W$. We have to show that $\langle [c_1(L|_W,\metr)]_E , \alpha \rangle \geq 0$. We proceed similarly as in the second part of the proof of Theorem \ref{psh slope condition}. We cover the support of $\alpha$ by finitely many non-empty tropical frames $(V_i,\varphi_{U_i},\Omega_i,s_i,\phi_i)$, $i=1, \dots, s$, satisfying  {(i'),} (ii') such that $\alpha$ is given on $V_i$ by $\alpha_i \in P^{n-1,n-1}(V_i,\varphi_{U_i})$. Then there is a unique $\alpha_U \in P^{n-1,n-1}(\Trop(U),\varphi_U)$ such that $\alpha$ is given on $\Uan$ by $\alpha_U$ (see \cite[Prop. 5.7]{gubler-kuennemann}). We use the same $\Omega$ and $\phi$ as in the proof of Theorem \ref{psh slope condition}. We have again $\supp(\alpha) \cap \Uan \subseteq V = \trop_U^{-1}(\Omega)$. Evaluating \eqref{candeco} at $\alpha$ gives 
\begin{equation} \label{ps3}
\langle [c_1(L|_W,\metr)]_E , \alpha \rangle = \int_{|\Trop(U)|} \dpa \dpb \phi \wedge \alpha_U  + \int_{|\Trop(U)|}\delta_{\tilde\phi\cdot \Trop (U)} \wedge \alpha_U .
\end{equation}
We have seen in Remark \ref{rem psh slope condition}(c),(d) that the tropical frame $(V,\varphi_U,\Omega,s,\phi)$ also satisfies (i') and (ii'). 
Since $\alpha$ is a positive $\delta$-form,  $\alpha_U$ is positive in $P(V,\varphi_U)$ (see Proposition \ref{local positivity is fine}). 
By Examples \ref{polyhedral supercurrent} and \ref{convex function}, $\dpa \dpb \phi$ restricts to a positive superform on $\relint(\Delta) \cap \Omega$ for any maximal face of $\Ccal$.  Then  Proposition \ref{positivity properties of superforms} and Example \ref{polyhedral supercurrent}  show that $ \dpa \dpb \phi \wedge \alpha_U$ induces a positive polyhedral supercurrent on $\Omega$ and hence the first integral in \eqref{ps3} is non-negative. 
Let $\tilde\Omega$ be an open subset of $N_{U,\R}$ with $\Omega =\tilde\Omega \cap \Trop(U)$. 
We choose a $\delta$-preform $\tilde\alpha_U \in P^{n-1,n-1}(\tilde\Omega)$ representing $\alpha_U$. By \cite[Prop. 1.14]{gubler-kuennemann}, we have 
$$\int_{|\Trop(U)|}\delta_{\tilde\phi\cdot \Trop (U)} \wedge \alpha_U = \int_{N_{U,\R}} \delta_{\tilde\phi \cdot N_{U,\R}} \wedge \delta_{\Trop(U)}  \wedge \tilde\alpha_U.$$
Since $\alpha_U$ is positive, we get immediately that the $\delta$-preform $\delta_{\Trop(U)}  \wedge \tilde\alpha_U$ on $\tilde\Omega$ is positive. Using (ii') and Proposition \ref{positivity properties of preforms}, we deduce that the above integral is non-negative  and hence \eqref{ps3} is non-negative.
\end{proof}

Let $L$ be a line bundle on $X$ equipped with a 
piecewise smooth metric $\metr$ over  {an} open subset $W$ of $\Xan$. 
Recall from \cite[9.9--9.11]{gubler-kuennemann} that  $\metr$  is called a {\it $\delta$-metric} if $c_1(L,\metr)$ is a well-defined $\delta$-form.

\begin{prop}\label{delta comparison thm}
The following properties are equivalent for a $\delta$-metric:
\begin{itemize}
 \item[(i)] The $\delta$-form $c_1(L|_W,\metr)$ is positive on $W$.
  \item[(ii)] The metric $\metr$ is functorial psh over $W$.
 \end{itemize}
\end{prop}

\proof
The formation of the first Chern $\delta$-form is compatible with pull-back.
Hence the equivalence of (i) and 
 (ii) follows directly from  {Proposition \ref{functorial current positive and positive delta-forms}}.
\qed

\begin{cor}\label{smooth comparison thm}
Let $L$ be a line bundle on a variety $X$ equipped with a 
smooth metric $\metr$ over the  open subset $W$ of $\Xan$. Then the following properties are equivalent:
\begin{itemize}
 \item[(i)] The {smooth form} $c_1(L|_W,\metr)$ is positive on $W$. 
 \item[(ii)] The metric $\metr$ is functorial $\delta$-psh over $W$.
 \item[(iii)] The metric $\metr$ is functorial psh over $W$.
 \item[(iv)] The metric $\metr$ is psh over $W$.
 \end{itemize}
\end{cor}

\proof The equivalence of (i) and (iii) is a  consequence of Proposition \ref{delta comparison thm} 
if one observes that  a smooth form is positive if and only if it is a positive $\delta$-form (see Remark \ref{compare cld}).
The remaining equivalences follow easily from Remark \ref{positivity for smooth forms and its current}. 
\qed

\begin{ex}\label{exampleabelianvarieties}
Let $A$ be an abelian variety over $K$. Let $L$ be a line bundle on
$A$ equipped with a canonical metric $\metr$. We investigate $c_1(L,\metr)$.
Hence we assume that our metric is induced by a rigidification of $L$
at zero.
\begin{enumerate}
\item[(a)]
If $L$ is ample, then $\metr$ is plurisubharmonic.
\item[(b)]
If $L$ is algebraically equivalent to zero, then the
$\delta$-form $c_1(L,\metr)$ vanishes.
\end{enumerate}
\end{ex}

\proof
We recall from \cite[Example 9.17]{gubler-kuennemann} that the canonical metric $\metr$ is a $\delta$-metric. The 
argument was based on \cite[Example 8.15]{gubler-kuennemann}, where it was shown that $\metr$ is locally with 
respect to the analytic topology equal to the product of a smooth metric with a piecewise linear metric. 
We recall some details from the proof of this local factorization. We used the Raynaud extension 
\begin{equation} \label{Raynaud extension 2}
1 \longrightarrow \Tan \longrightarrow E \stackrel{q}{\longrightarrow} 
B^{\rm an} \longrightarrow 0,
\end{equation}
where $T$ is a 
multiplicative torus of rank $r$  and $B$ is an abelian variety of good reduction. There is an analytic quotient 
map $p:E \to A^\an$ which is locally an isomorphism. There is a canonical map $\val:E \to N_\R$, where $N$ is the 
cocharacter lattice of $T$. The cocycles of the line bundle $L$ determine a canonical quadratic function $q_0:N_\R \to \R$ 
with associated symmetric bilinear form 
$b:N_\R\times N_\R\longrightarrow \R$. Then we have  
$p^*\metr := e^{-q_0 \circ \val}{q^*}\metr_\Hcal$ 
on $p^*(\Lan)$. Here, $\Hcal$ is a line bundle on the abelian $\kcirc$-scheme $\Bcal$ with generic fibre $B$ and hence ${q^*}\metr_\Hcal$ 
is a  piecewise linear metric. Moreover, there is a unique metric $\metr_{\rm sm}$ on $O_E$ with $\|1\|_{\rm sm}=e^{-q_0 \circ \val}$. Since $\metr_{\rm sm}$ is a smooth metric and $p$ is a local isomorphism, 
this proves the desired local factorization.

In case (a), the symmetric bilinear form $b$ is positive definite and hence the first Chern form of $\metr_{\rm sm}$ is a positive smooth form. Moreover, the line bundle $\Hcal$ is again ample. Both claims are proved in \cite[Thm. 6.13]{bosch-luetkebohmert-1991}. 
Hence ${q^*}\metr_\Hcal$ is a semipositive piecewise linear metric and $\metr$ is locally the product of a smooth metric with positive first Chern form and a semipositive piecewise linear metric. By Corollary \ref{algebraic vs locally formal} and Theorem \ref{localnotes5}, the latter has a positive first Chern $\delta$-form and hence the first Chern $\delta$-form of the canonical metric $\metr$ on $L$ is positive as well. 

In case (b), the symmetric bilinear form $b$ is zero 
(see comments before \cite[Thm. 6.8]{bosch-luetkebohmert-1991}). 
Hence $q_0$ is linear and hence $d'd''q_0=0$. 
This means that the first Chern form of $\metr_{\rm sm}$ is zero. 
Moreover, the line bundle $\Hcal$ is algebraically equivalent to $0$. 
We conclude that ${q^*}\metr_\Hcal$ is a semipositive piecewise linear metric. 
Since semipositivity is a local analytic property, we deduce that $\metr$ is locally the product of a smooth metric with zero Chern form and of a semipositive piecewise linear metric. 
By Corollary \ref{algebraic vs locally formal} and 
Theorem \ref{localnotes5}, we see that $c_1(L,\metr)$ is 
a positive $\delta$-form. 
The same argument shows that $c_1(L^{-1},\metr)=-c_1(L,\metr)$ 
is a positive $\delta$-form.
We get $c_1(L,\metr)=0$ from 
Lemma \ref{positive and negative generalized delta-form}.
\qed

\begin{rem} \label{canonical metric on algebraic equivalent to zero}
Let $L$ be a line bundle on a proper smooth variety over $K$ which is algebraically equivalent to zero. We have seen in \cite[Example 8.16]{gubler-kuennemann} that a canonical metric $\metr_{\rm can}$ on $L$ is a $\delta$-metric as a positive tensor power is piecewise linear. Since the canonical metric is obtained by pull-back from a canonical metric on an odd line bundle on an abelian variety (see \cite[Example 8.16]{gubler-kuennemann}), we deduce from Example \ref{exampleabelianvarieties} that $c_1(L,\metr_{\rm can})=0$. 
\end{rem}

\begin{prop} \label{Zariski dense open subset}
Let $L$ be a line bundle on a variety $X$ over $K$ and let $U$ be a dense Zariski open subset of $X$. We consider a piecewise smooth metric $\metr$ on $L$ over an open subset $W$ of $\Xan$. Let $(P)$ be one of the four properties: psh, functorial psh, $\delta$-psh, functorial $\delta$-psh. 
Then $\metr$ has property $(P)$ if and only if the restriction of $\metr$ to $L^{\rm an}|_{\Uan \cap W}$ fulfills $(P)$.
\end{prop}

\begin{proof}
The preimage of a Zariski dense open subset is again a Zariski dense open subset and so the functoriality assertions follow from the corresponding assertions on $X$. 
The restriction of a psh (resp. $\delta$-psh) metric over $W$ to any open subset of $W$ is obviously psh (resp. $\delta$-psh). Conversely, assume that the restriction of $\metr$ to $L^{\rm an}|_{\Uan \cap W}$ is psh (resp. $\delta$-psh). 
Let $n:=\dim(X)$ and let $\alpha \in P^{n-1,n-1}(W)$. Using \ref{canonical decomposition}, we note that $\beta:= c_1(L^{\rm an}|_W,\metr)_{\rm res} \wedge \alpha$ and $ \omega:=c_1(L^{\rm an}|_W,\metr)_{\rm ps} \wedge \alpha$ have both support in $\Uan \cap W$ by \cite[Cor. 5.6, 9.5]{gubler-kuennemann}. Now assume that $\alpha$ is a positive smooth form (resp. a positive  $\delta$-form) with compact support in $W$. Then $\beta$ and $\omega$ have both compact support in $\Uan \cap W$.  
Using partition of unity \cite[Cor. 3.3.4]{chambert-loir-ducros}, there is a smooth function $\phi \geq 0$ with compact support in $W$ such that $\phi$ is identically $1$ on the supports of $\beta$ and $\omega$. Since $\phi\alpha$ is positive on $W \cap \Uan$, we get from \ref{candeco}  
$$\langle [c_1(L^{\rm an}|_W,\metr)], \alpha \rangle = \langle [c_1(L^{\rm an}|_{\Uan \cap W},\metr)], \phi\alpha \rangle \geq 0$$
using currents (resp. $\delta$-currents). This proves that $\metr$ is psh (resp. $\delta$-psh).
\end{proof}

Let $X$ be a  toric variety over $K$ with dense open torus $T$ and let $\metr$ be a  piecewise smooth toric metric  {on a} line bundle  $L$ on $X$. A toric section $s$ of $L$ 
(i.e.\ a rational section  invertible over $T$)
induces a  function $\phi$ on $N_\R$ with 
\begin{equation} \label{pseq}
\phi \circ \trop_T = -\log \| s \|
\end{equation}
on $\Tan$, where $N$ is the cocharacter lattice of $T$. {Note that $\phi$ is locally a piecewise smooth function.}

\begin{prop} \label{torus prop}
For a line bundle $L$ on a toric variety $X$, we assume that $\metr$ is a piece\-wise smooth toric metric as above. Then the following properties are equivalent.
\begin{itemize}
\item[(i)] The metric $\metr$ is functorial $\delta$-psh.
\item[(ii)] The metric $\metr$ is functorial psh.
\item[(iii)] The metric $\metr$ is psh.
\item[(iv)] The function $\phi$ from \eqref{pseq} is convex.
\end{itemize}
\end{prop}

\proof By Proposition \ref{Zariski dense open subset}, we may assume that $X=T$. It is clear that (i) yields (ii) and that (ii) yields (iii). By Theorem \ref{psh slope condition}, property (iii) implies (iv). Finally, assume that $\phi$ is convex. It follows from Examples \ref{polyhedral supercurrent} and \ref{convex function} that the assumptions from Proposition \ref{second psh slope condition} are satisfied and hence (i) holds. 
\qed

\begin{cor} \label{torus cor}
We assume that $\metr$ is a toric formal metric on  {a} line bundle $L$ of a toric variety $X$ over $K$. Then there is a unique  piecewise linear function $\phi: N_\R \to \R$ with \eqref{pseq}. Moreover, 
the formal metric $\metr$ is semipositive if and only if $\phi$ is convex.
\end{cor}

\begin{proof} It follows from \ref{non-reduced analytic space} that a metric is formal if and only if it is piecewise linear. This proves the first claim easily. For a formal metric, Theorem \ref{localnotes5} shows that semipositive is equivalent to functorial psh. Now the final claim follows from Proposition \ref{torus prop}. 
\end{proof}
This corollary is important for the characterization of all toric continuous metrics on a proper toric variety over $K$ given in the paper of Burgos--Philippon--Sombra over discretely valued fields 
\cite{burgosetal-2011} and generalized in the 
 {thesis of Julius Hertel 
\cite{hertel-thesis, gubler-hertel-2015}.}

\def\cprime{$'$}

\end{document}